\documentclass[12pt, reqno]{amsart}
\usepackage{amssymb,amsthm,amsfonts,amsmath, amscd}

\usepackage{hyperref}
\usepackage{mathrsfs}
\usepackage{color}

\newcommand\C{\mathbb C}
\newcommand\R{\mathbb R}
\newcommand\Z{\mathbb Z}

\newcommand\EE{\mathbb E}

\newcommand\om{\omega}
\newcommand\al{\alpha}
\newcommand\be{\beta}

\newcommand\Ga{\Gamma}
\newcommand\de{\delta}

\newcommand\la{\lambda}

\newcommand\Om{\Omega}
\newcommand\si{\sigma}
\newcommand\Si{\Sigma}
\newcommand\epsi{\varepsilon}

\newcommand\PPsi{\overline{\Psi}}

\newcommand\Cliff{\operatorname{Cliff}}
\newcommand\const{\operatorname{const}}

\newcommand\sgn{\operatorname{sgn}}
\newcommand\Pf{\operatorname{Pf}}
\newcommand\Ker{\operatorname{Ker}}

\newcommand\CAR{\operatorname{CAR}}
\newcommand\GICAR{\operatorname{GICAR}}

\newcommand\tr{\operatorname{tr}}
\newcommand\Gr{\operatorname{Gr}}

\newcommand\dHermite{{\operatorname{dHermite}}}

\newcommand\dSine{{\operatorname{dSine}}}
\newcommand\detreg{\operatorname{det}_2}

\newcommand\J{P^{(a,a)}}

\renewcommand\SS{\mathcal S}

\renewcommand\AA{\mathcal A}

\newcommand\FF{\mathscr F}
\newcommand\PP{\mathcal P}
\newcommand\T{\mathcal T}

\newcommand\X{\mathfrak X}
\newcommand\A{\mathfrak A}

\newcommand\one{\mathbf1}
\newcommand\K{\mathbf{K}}

\newcommand\wt{\widetilde}

\newcommand\ccdot{\,\cdot\,}

\newtheorem{theorem}{Theorem}[section]
\newtheorem{proposition}[theorem] {Proposition}
\newtheorem{corollary}[theorem]{Corollary}
\newtheorem{lemma}[theorem]{Lemma}

\theoremstyle{definition}
\newtheorem{definition}[theorem]{Definition}
\newtheorem{remark}[theorem]{Remark}

\numberwithin{equation}{section}

\begin{document}

\title[]
{Determinantal point processes and \\
fermion quasifree states}

\author{Grigori Olshanski}

\address{$^1$ Institute for Information Transmission Problems, Moscow, Russia; 
\newline
\indent $^2$ Skolkovo Institute of Science and Technology, Moscow, Russia; 
\newline
\indent $^3$ National Research University Higher School of Economics, Moscow, Russia
\newline
\indent \emph{Email}: olsh2007@gmail.com}

\date{}

\begin{abstract}
Determinantal point processes are characterized by a special structural property of the correlation functions: they are given by minors of a correlation kernel. However, unlike the correlation functions themselves, this kernel is not defined intrinsically, and the same determinantal process can be generated by many different kernels. The non-uniqueness of a correlation kernel  causes difficulties in studying determinantal processes. 

We propose a formalism which allows to find a distinguished correlation kernel under certain additional assumptions. The idea is to exploit  a connection between determinantal processes and quasifree states on CAR, the algebra of canonical anticommutation relations.

We prove that the formalism applies to discrete $N$-point orthogonal polynomial ensembles and to some of their large-$N$ limits including the discrete sine process and  the determinantal processes with the discrete Hermite, Laguerre, and Jacobi kernels investigated by Borodin and the author in \cite{BO-2017}. As an application we resolve the equivalence/disjointness dichotomy for some of those processes. 
\end{abstract}

\maketitle

\tableofcontents

\section{Introduction}\label{sect1}

Let $\X$ be a countable set and $\Om$ denote the product space $\{0,1\}^\X$. To any collection  $\{p_x: x\in\X\}$ of real numbers, such that $0< p_x<1$, there corresponds a product measure on $\Om$, where $p_x$ defines the probability of $1$ at a given position $x\in\X$. The following claim is a particular case of classic Kakutani's theorem (1948) (see \cite[pp. 2018 and 2022]{Kakutani}).

\smallskip

\noindent\textbf{Theorem.} \emph{Two product measures on $\Om$, which correspond to two collections $\{p_x\}$ and $\{p'_x\}$ as above, are either equivalent or disjoint\footnote{Two measures are said to be \emph{disjoint} (or \emph{mutually singular} or else \emph{orthogonal}) if there exists a measurable set $A$ such that one measure is supported by $A$ while the other is supported by the complement of $A$. }, depending on whether the series
$$
\sum_{x\in\X}((\sqrt{p_x}-\sqrt{p'_x})^2+(\sqrt{1-p_x}-\sqrt{1-p'_x})^2) 
$$
is convergent or divergent. If both collections are separated from $0$ and $1$, then the above series may be replaced by a simpler one, $
\sum_{x\in\X}(p_x-p'_x)^2.
$
} 

\smallskip

One of the motivations of the present paper comes from the following open problem: find conditions of  equivalence and disjointness for a more general class of probability measures ---  the determinantal measures.

\subsection{Determinantal measures} 

Let $\X$ and $\Om$ be as above.  Elements $\om\in\Om$ are functions $\om(x)$ on $\X$ taking the binary values $0$ and $1$. Alternatively, we regard each $\om$ as a subset of $\X$, by identifying every subset with its indicator function. The space $\Om$ is compact in the product topology and is homeomorphic to the Cantor set.  

Let $\PP(\Om)$ denote the space of probability Borel measures on $\Om$. Any measure $M\in\PP(\Om)$ is uniquely determined by the infinite collection $\rho_1,\rho_2,\dots$ of \emph{correlation functions}. Here, for every $n=1,2,\dots$, $\rho_n$ is a symmetric function of $n$ arguments $x_1,\dots,x_n\in\X$, assumed to be pairwise distinct, and the value $\rho_n(x_1,\dots,x_n)$ equals the mass assigned by $M$ to the cylinder set 
\begin{equation}\label{eq1.A1}
\{\om\in\Om:\; \om(x_1)=\dots=\om(x_n)=1\}.
\end{equation}

\begin{definition}\label{def1.A}
We say that $M\in\PP(\Om)$ is a \emph{determinantal measure} if there exists a complex valued function $K(x,y)$ on $\X\times\X$ such that for every $n$ and any $n$-tuple $(x_1,\dots,x_n)$ as above, 
\begin{equation}\label{eq1.A2}
\rho_n(x_1,\dots,x_n)=\det[K(x_i,x_j)]_{i,j=1}^n.
\end{equation}
Then $K(x,y)$ is called a \emph{correlation kernel} of $M$.
\end{definition}

The definition of determinantal measures can be extended to more general, not necessarily discrete, spaces $\X$, but in the present paper we focus on the discrete case, which is already rich enough. For more details about determinantal measures see Ben Hough, Krishnapur, Peres, and Virag \cite{BKPV}, Borodin \cite{Bor-2011}, Lyons \cite{Lyons}, Shirai and Takahashi \cite{ST-1}, \cite{ST-2}, Soshnikov \cite{Soshnikov}.\footnote{A note about terminology: for the purposes of the present paper one may not distinguish a random process from its distribution; for this reason  we abandon the conventional term ``determinantal point processes'' and speak about ``determinantal measures''.}

Let $\ell^2(\X)$ be the complex Hilbert space with a distinguished orthonormal basis $\{e_x\}$ indexed by the set $\X$.\footnote{We keep to the convention that the inner product is linear in the first variable.}  An operator  $K$ on $\ell^2(\X)$ is called a \emph{positive contraction} if $K=K^*$ and $0\le K\le1$.  It is known (see the papers cited above) that for any positive contraction $K$, its matrix $K(x,y):=(K e_y,e_x)$ serves as a correlation kernel for a (necessarily unique) determinantal measure. That measure will be denoted by $M^K$. 

In particular, the measure $M^K$ is defined for any \emph{projection operator} $K$, that is, the operator of orthogonal projection onto a subspace in $\ell^2(\X)$. The class of determinantal measures with projection correlation kernels is important because it embraces a lot of concrete examples and because the projection property is substantially used in a number of results, see Bufetov \cite{Buf-2018}, Lyons \cite{Lyons}. 
  
\subsection{Quasifree states}\label{sect1.2}
Let $\A=\A(\X)$ denote the unital $C^*$-algebra generated by the elements $a^+_x, a^-_x$ indexed by the points $x\in\X$, with the defining relations
$$
a^+_xa^+_y+a^+_ya^+_x=0, \quad a^-_xa^-_y+a^-_ya^-_x=0, \quad a^+_xa^-_y+a^-_ya^+_x=\de_{xy},  \quad (a^+_x)^*=a^-_x, \qquad x,y\in\X,
$$
and let $\A^0=\A^0(\X)$ be its $C^*$-subalgebra generated by the elements of the form $a^+_xa^-_y$.  The algebra $\A$ is the \emph{algebra of canonical anticommutation relations} (CAR, for short), associated with the complex Hilbert space $\ell^2(\X)$, and $\A^0\subset\A$ is called the \emph{gauge invariant subalgebra} (some authors use for it the abbreviation GICAR). We use a bit nonstandard notation ($a^+_x, a^-_x$ instead of $a^*_x, a_x$), as in Meyer \cite[ch. II, \S5]{Meyer}. 

\begin{definition}\label{def1.B}
Let $K$ be a positive contraction on $\ell^2(\X)$ and $K(x,y)=(K e_y,e_x)$ be its matrix, as before. The \emph{quasifree state} on $\A^0$ corresponding to $K$ is the linear functional $\varphi[K]: \A^0\to\C$, uniquely defined by the following conditions. First, $\varphi[K](1)=1$. Second, for any $n=1,2,\dots$ and any two $n$-tuples $x_1,\dots,x_n$ and $y_1,\dots,y_n$,
\begin{equation}\label{eq1.A3}
\varphi[K](a^+_{x_n}\dots a^+_{x_1}a^-_{y_1}\dots a^-_{y_n})=\det[(K(y_i,x_j)]_{i,j=1}^n.
\end{equation}
\end{definition}

It is known that $\varphi[K]$ is indeed a state, that is, $\varphi[K](aa^*)\ge0$ for any $a\in\A^0$. A similar definition holds for the $\CAR$ algebra $\A\supset \A^0$: then one adds the condition
$$
\varphi[K](a^+_{x_m}\dots a^+_{x_1}a^-_{y_1}\dots a^-_{y_n})=0, \quad m\ne n. 
$$

The quasifree states on $\CAR$ were intensively studied in the sixties, see Powers--St{\o}rmer \cite{PS-1970}, Araki \cite{Araki},  and references therein (and also the book Bratteli--Robinson \cite{BR}). The more delicate case of $\GICAR$ was investigated later, see Baker \cite{Baker}, Str\u{a}til\u{a}--Voiculescu \cite{SV-1978}.  

\subsection{The interplay between determinantal measures and quasifree states} 
Comparing the above two definitions, we immediately see that they are very similar. There is in fact not only  a formal similarity, but a direct connection, which we describe now. The key observation is that the algebra $C(\Om)$ of continuous functions on $\Om$ with pointwise operations can be identified, in a natural way, with the maximal commutative subalgebra of $\A^0$ generated by the quadratic elements of the form $a^+_xa^-_x$, $x\in\X$. Under this identification, the indicator of the cylinder set \eqref{eq1.A1} is identified with the element
\begin{equation}\label{eq1.A4}
a^+_{x_n}\dots a^+_{x_1}a^-_{x_1}\dots a^-_{x_n}=(a^+_{x_1}a^-_{x_1})\dots(a^+_{x_n}a^-_{x_n}).
\end{equation}

Let, as above, $K$ be a positive contraction and $\varphi[K]\!\downarrow\!C(\Om)$ denote the restriction of the quasifree state $\varphi[K]$ to $C(\Om)$. Recall a well-known general fact: there is a natural one-to-one correspondence between $\PP(\Om)$ and states on $C(\Om)$: namely, the state corresponding to a measure $M\in\PP(\Om)$ is the expectation $\EE_M$ defined by
\begin{equation}\label{eq1.B}
\EE_M(f):=\int_\Om f(\om) M(d\om), \qquad f\in C(\Om).
\end{equation}
Then we see from \eqref{eq1.A2}, \eqref{eq1.A3}, and \eqref{eq1.A4} that the measure corresponding to  the state $\varphi[K]\!\downarrow\!C(\Om)$ is nothing else than the determinantal measure $M^K$.  
In this way we obtain a natural correspondence 
\begin{equation}\label{eq1.A5}
\text{quasifree states} \longrightarrow \text{determinantal measures}, \qquad \varphi[K]\mapsto M^K.
\end{equation}

The correspondence \eqref{eq1.A5} is an example of a link between noncommutative probability (states on a noncommutative algebra) and classical probability (probability measures, or states on a commutative algebra).\footnote{Note that, in the literature, there are other examples of such a kind, but of a different nature, see e.g. Biane \cite{Biane}.} The goal of the present paper is to better understand the interplay between determinantal measures and quasifree states, and to apply it to the study of determinantal measures. 

Note a peculiarity of the definition of determinantal measures via formula \eqref{eq1.A2}: the correlation functions $\rho_n$ on the left-hand side are invariants of $M$, but the kernel $K(x,y)$ on the right-hand side is not. Indeed, the kernel is represented in \eqref{eq1.A2} by its diagonal minors only, which do not suffice for its reconstruction. For instance, any ``gauge transformation''  of the form
\begin{equation}\label{eq1.C}
K(x,y) \leadsto f(x)K(x,y)f(y)^{-1}
\end{equation}
does not affect the correlation functions. Even in the class of real-valued projection kernels there is the freedom to take as $f$  an arbitrary function with values $\pm1$.\footnote{In the class of symmetric kernels $K(x,y)$, the gauge transformations with $f(x)=\pm1$ are the only transformations that preserve the diagonal minors, see Stevens \cite{Stevens}. } As a consequence, it turns out that, typically, there are many different kernels defining one and the same measure.  It also follows that the correspondence $\varphi\mapsto M$ in \eqref{eq1.A5} is typically many-to-one. 

The idea of the present paper is to try to invert the correspondence $\varphi\mapsto M$, that is, to find a way to single out  a ``canonical'' correlation kernel for $M$. We show that this is possible, at least in some cases of interest.  The key condition that we need is that $M$ has to be \emph{quasi-invariant} with respect  to the action of a natural countable group of transformations of $\Om$. In \cite{Ols-2011}, the desired quasi-invariance property was established for a particular family of determinantal measures originated from asymptotic representation theory. Then Bufetov \cite[Theorem 1.6]{Buf-2018} showed, by a different method, that this property is not an exceptional phenomenon --- it holds for a broad class of measures with projection kernels. Thus, our key condition seems to be reasonable and not too restrictive. 

Lytvynov \cite{Lyt} and Lytvynov-Mei \cite{LytM} earlier showed that a version of the correspondence $\varphi\mapsto M$ also holds in the continuous case, when $\X$ is a locally compact topological space and as $\Om$ one takes the space of locally finite point configurations in $\X$. However, for continuous spaces, the definition of the correspondence $\varphi\mapsto M$ is not so obvious as  for discrete ones. The reason is that the creation/annihilation operators $a^\pm_x$ parameterized by points of $\X$ cannot be defined as elements of the algebra $\CAR$ --- these ``generalized elements'' need be smoothed; but even then there is a problem with giving meaning to the products $a^+_xa^-_x$. As a result, there is no direct analogue of the embedding $C(\Om)\hookrightarrow\A^0$. In the papers  \cite{Lyt}, \cite{LytM}, this difficulty was successfully overcome, but this required considerable work. Whether it is possible to extend our formalism to the continuous case is not yet clear. 

\subsection{The equivalence/disjointness problem}

As shown in  Str\u{a}til\u{a}--Voiculescu \cite[\S3]{SV-1978} (see also Baker \cite{Baker}), the following results hold for  quasifree states  $\varphi[K]$ of $\GICAR$:

\smallskip

$\bullet$ $\varphi[K]$ is a pure state (that is, the corresponding cyclic representation of $\GICAR$ is irreducible) if and only if $K$ is a projection operator;

\smallskip

$\bullet$  two pure quasifree states, $\varphi[K_1]$ and $\varphi[K_2]$ are equivalent if and only if the difference  $K_1-K_2$ is of Hilbert--Schmidt class and some Fredholm operator built from $K_1$ and $K_2$ has index $0$.

\smallskip   

Note that the Hilbert--Schmidt condition also plays a key role in some other equivalence criteria: for boson quasifree states and for Gaussian measures.   

As was pointed out in author's paper \cite{Ols-2011}, no similar general criterion for determinantal measures is known.
In that paper, the following problem was posed:

\smallskip

\emph{Problem.} Assume we are given two determinantal measures on a common space. How to test their equivalence (or, on the contrary, disjointness)? Is it possible to decide this by inspection of their correlation kernels?

\smallskip

One could imagine that equivalence of determinantal measures is somehow related to closeness of their kernels in an appropriate sense. However, the non-uniqueness of correlation kernels is a source of difficulty here. For instance, the difference between a kernel and its modification by a ``gauge transformation'' \eqref{eq1.C} can be large enough (in particular, not in the  Hilbert--Schmidt class), while any such modification does not affect the measure at all.

We show that the construction of ``canonical'' correlation kernels, when applicable, removes this difficulty and allows to resolve the problem of equivalence/disjointness for certain concrete determinantal measures.

\subsection{Organization of the paper and summary of results}

\subsubsection{From quasi-invariant measures to states on $\A^0$}

Let, as above, $\X$ be a countable set, $\Om=\{0,1\}^\X$, and $\PP(\Om)$ be the set of probability Borel measures on $\Om$. Next, let $\SS$ denote the group of \emph{finitary permutations} of $\X$ (a permutation $g:\X\to\X$ is said to be finitary if $g(x)\ne x$ for finitely many $x\in\X$). The group $\SS$ is countable; it  acts, in a natural way, on the space $\Om$ by homeomorphisms. This makes it possible to define the cross product $C^*$-algebra $C(\Om)\rtimes\SS$. 

First, in Section \ref{sect2}, we recall the well known procedure which assigns to an arbitrary $\SS$-quasi-invariant measure $M\in\PP(\Om)$ a representation $\T[M]$ of the algebra $C(\Om)\rtimes\SS$ acting on the Hilbert space $L^2(\Om,M)$. The representation $\T[M]$ is irreducible if and only if $M$ is ergodic. 

Next, in Sections \ref{sect3} - \ref{sect4}, we suppose that $\X$ is endowed with a linear order with finite intervals. Under this assumption we construct a surjective homomorphism $C(\Om)\rtimes\SS\to\A^0$ and show that $\T[M]$ is factored through it, so that $\T[M]$ can be treated as a representation of the algebra $\A^0$; that representation will be denoted by $T[M]$. 

Let us denote by $\tau[M]$ the state on $\A^0$ defined by
$$
\tau[M](a):=(T[M](a)\one,\one)_{L^2(\Om,M)}, \qquad a\in\A^0,
$$
where $\one$ stands for the function on $\Om$ identically equal to $1$. By the very definition, $\tau[M]\downarrow C(\Om)$ is the state $\EE_M$ defined by \eqref{eq1.B}. Thus,  $\tau[M]$ keeps the whole information about the initial measure $M$. Because the noncommutative algebra $\A^0$ possesses  a richer structure than the commutative algebra $C(\Om)$, it is tempting to apply the correspondence $M\mapsto \tau[M]$ to the study of determinantal measures. But for this we need to understand the nature of the states $\tau[M]$. This leads us to the question: let $M$ be an $\SS$-quasi-invariant measure; when is $\tau[M]$ a quasifree state? In other words, when $\tau[M]=\varphi[K]$ for some $K$? 

It is convenient to give a name to measures $M$ with this property. We will say that $M$ is \emph{perfect} if $\tau[M]$ is a quasifree state $\varphi[K]$. Then $M$ is automatically a determinantal measure and we will call $K(x,y)$ (the matrix of $K$) the \emph{canonical} correlation kernel of $M$. 

Note that product measures are not perfect, see Remark \ref{rem4.D}. In the second part of the paper we provide examples of perfect measures. 

\subsubsection{Discrete orthogonal polynomials ensembles}
In Section \ref{sect5} we assume that $(\X,<)$ is a countable or finite\footnote{All the results stated above for countable sets hold for finite sets as well.} subset of $\R$ with the order induced from $\R$, $N$ is a fixed positive integer, and $W(x)$ is a strictly positive function on $\X$ such that 
$$
\sum_{x\in\X}x^{2N} W(x)<\infty.
$$
Denote by $\Om_N\subset\Om$ the set of $N$-point subsets of $\X$.  The \emph{$N$-point orthogonal polynomial ensemble with weight function $W$} is determined by the probability measure $M_{N,W}\in\PP(\Om)$ supported by $\Om_N$ and given by 
$$
M_{N,W}(\om)=\frac1{\mathcal Z_N}\cdot\prod_{i=1}^N W(x_i)\cdot \prod_{1\le i<j\le N} (x_i-x_j)^2, \qquad \om=\{x_1,\dots,x_N\}\in\Om_N,
$$
where $\mathcal Z_N$ is the normalizing constant. 

Let $\K_{N,W}$ be the operator of orthogonal projection onto the $N$-dimensional subspace in $\ell^2(\X)$ spanned by the functions $x^n W^{1/2}(x)$, $n=0,1,\dots,N-1$, and let $K_{N,W}(x,y)=(\K_{N,W}e_y,e_x)$ be the corresponding kernel on $\X\times\X$. It is well known (K\"onig  \cite{Konig}) that $M_{N,W}$ is a determinantal measure admitting $K_{N,W}(x,y)$ as a correlation kernel.  

Because the set $\Om_N$ forms an $\SS$-orbit and the measure $M_{N,W}$ charges every $\om\in\Om_N$, the measure is $\SS$-quasi-invariant. Therefore, the results of Section \ref{sect4} are applicable and the state $\tau[M_{N,W}]$ on $\A^0$ is well defined. 

Theorem \ref{thm5.A} states that $\tau[M_{N,W}]=\varphi[\K_{N,W}]$. Thus, the measure $M_{N,W}$ is a perfect measure in the sense of Definition \ref{def4.A}, and its canonical correlation kernel is $K_{N,W}(x,y)$.

Many concrete examples of discrete orthogonal polynomials satisfying our assumptions are provided by the Askey scheme and its $q$-analogue (Koekoek--Lesky--Swarttouw \cite{KLS-2010}). 

In Theorem \ref{thm5.A}, it is not too surprising that the canonical kernel turns out to be the most natural one, but the very fact that $M_{N,W}$ is perfect is not evident. Our proof relies on an elegant formula for the average of products of characteristic polynomials contained in Proposition 4.1 of the paper \cite{StrFyo} by Strahov and Fyodorov. 

\subsubsection{Limits of discrete orthogonal polynomial ensembles}

In the theory of determinantal point processes, orthogonal polynomial ensembles (discrete and continuous) play a fundamental role, because lots of less elementary determinantal processes are obtained from them via various large-$N$ limit transitions. Some structural properties persist in limit transitions: for instance, the so-called integrable form of various limit kernels comes from the Christoffel--Darboux identity for orthogonal polynomials. 

In connection with this, a natural question arises: does the ``perfectness'' property survive in large-$N$ limit transitions?

In Theorem \ref{thm8.A} it is shown that under some technical conditions the answer is positive. This result is deduced  from Theorem \ref{thm5.A} after a preparation occupying Sections \ref{sect6} and \ref{sect7}. The material of these two auxiliary sections is based on some results extracted from Lyons \cite{Lyons} and Bufetov \cite{Buf-2018}.  

From Theorem \ref{thm8.A} we obtain concrete examples of perfect measures with canonical projection correlation kernels corresponding to linear subspaces of infinite dimension and codimension, see Theorem \ref{thm8.B} and Theorem \ref{thm8.C}.  These kernels are the \emph{discrete Hermite, Laguerre, and Jacobi kernels} on $\Z_{\ge0}\times\Z_{\ge0}$ introduced in \cite{BO-2017} and the \emph{discrete sine kernel} on $\Z\times\Z$ defined by
$$
K^\dSine_\phi(x,y)=\begin{cases} \dfrac{\sin( \phi\,(x-y))}{\pi(x-y)},  & x\ne y,\\
\phi/\pi,  & x=y \end{cases}
$$
(here $\phi\in(0,\pi)$ is the parameter). The discrete sine kernel first emerged in the study of the asymptotics of the Plancherel measures on partitions \cite{BOO}. Like the famous sine kernel on $\R\times\R$, the discrete sine kernel possesses a universality property \cite{BKMM-2007}.

\subsubsection{Applications}

The results described above, combined with the criterion of equivalence/disjointness of quasifree states (we recall it in Section \ref{sect9}), can be applied to the problem of equivalence/dis\-joint\-ness for determinantal measures. In Section \ref{sect10} we give two concrete examples. 

In Theorem \ref{thm10.A}, we consider the one-parameter family of discrete Hermite kernels and show that the corresponding determinantal measures are pairwise disjoint. 

In Theorem \ref{thm10.B} we deal with a one-parameter family of discrete Jacobi kernels and show that the corresponding determinantal measures are pairwise equivalent.

\section{Quasi-invariant measures and crossed products}\label{sect2}

In this section we describe a simple general construction that assigns to a discrete group $G$ and a $G$-quasi-invariant measure $M$ a representation $\T[M]$ of a crossed product $C^*$-algebra. We show how properties of measures are related to properties of the associated representations. 

\subsection{A dichotomy for ergodic measures}

In this subsection we fix a Borel space $(\Om,\Si)$, that is, $\Om$ is a set and $\Si$ is a $\si$-algebra of subsets of $\Om$. All measures are assumed to be defined on $\Si$, positive, and $\si$-finite. 
Given a measure $M$, we say that a set $A\in\Si$ is \emph{$M$-null} if $M(A)=0$; in this case the complement $\Om\setminus A$ is said to be \emph{$M$-conull}. 

Next, let $G$ be a countable (or finite) group of automorphisms of $(\Om,\Si)$.  It acts, in a natural way, on the set of measures. Namely, the transformation of a measure $M$ by an element $g\in G$ is the measure ${}^g\! M$ defined by
$$
{}^g\!M(A)=M(g^{-1}(A)), \qquad A\in\Si.
$$

A measure $M$ is said to be \emph{$G$-quasi-invariant} if, for any $g\in G$ the measures $M$ and ${}^g\!M$ are equivalent. This holds if and only if the collection of  $M$-null subsets is stable under the transformations from $G$. 

Recall that a $G$-quasi-invariant measure $M$ is said to be \emph{ergodic} if any $G$-invariant  set $A\in\Si$ is either $M$-null or $M$-conull.   

Because $G$ is at most countable, in the above condition one can require equally well that $A$ be $G$-invariant \emph{mod $0$}, meaning that, for any $g\in G$,  the symmetric difference $A\triangle g(A)$ is a $M$-null set. 

Two measures $M_1$ and $M_2$ are said to be \emph{disjoint} (or \emph{mutually singular}) if there exist two nonintersecting subsets $A_1, A_2\in\Si$ such that $A_1$ is $M_1$-conull and $A_2$ is $M_2$-conull. 

\begin{proposition}\label{prop2.A}
Let, as above, $(\Om,\Si)$ be a Borel space and $G$ be a countable or finite group of its automorphisms. 
For any two nonzero $G$-quasi-invariant ergodic measures  $M_1$ and $M_2$, the following dichotomy holds: $M_1$ and $M_2$ are either equivalent or disjoint.  
\end{proposition}

\begin{proof}
Suppose that $M_1$ and $M_2$ are not equivalent. This means that there exists a subset $A\in\Si$, which is a null set for one measure but not for the other. Let, for definiteness, $M_1(A)=0$ and $M_2(A)>0$. Then the same holds with $A$ replaced by $B:=\bigcup_{g\in G}g(A)$. Since $B$ is $G$-invariant, it is conull with respect to $M_2$. Therefore, $M_2$ is supported by $B$ while $M_1$ is supported by $\Om\setminus B$, so that $M_1$ and $M_2$ are disjoint. 
\end{proof}

This simple reasoning is taken from Yamasaki \cite[p. 147, Remark 3]{Y}. (After a modification it works for not necessarily countable group actions as well, see \cite[p. 144, Theorem 6.1]{Y}.)

\subsection{Representations associated with quasi-invariant measures}\label{sec2.2}
Let us recall the crossed product construction for $C^*$-algebras; for more details, see Brown--Ozawa \cite{BrownOzawa}, Williams \cite{Williams}. 

\begin{definition}\label{def2.A}
Let $\AA$ be a separable unital $C^*$ algebra and $G$ be a finite or countable group of its automorphisms. 

(i) A \emph{covariant representation\/} of $(\AA, G)$ is a
pair $(\T_1,\T_2)$, where $\T_1$ is a representation of $\AA$ and
$\T_2$ is a unitary representation of $G$ on a common separable Hilbert space, such
that
$$
\T_2(g)\T_1(f)\T_2(g^{-1})=\T_1({}^g\!f), \qquad f\in\mathcal A, \quad
g \in G,
$$
where ${}^g\!f$ denotes the result of application of $g$ to $f$. 

(ii) There exist a (unital and separable) $C^*$-algebra $\mathcal B$ equipped with a morphism $\AA\to\mathcal B$ and a morphism of $G$ into the unitary group of $\mathcal B$
such that:

1) these two morphisms are consistent (in a natural sense) with the action of
$G$ on $\AA$;

2) the images of $\AA$ and $G$ generate $\mathcal B$;

3) any covariant representation $(\T_1,\T_2)$ of $(\AA,G)$ factors
through a representation of $\mathcal B$.

Moreover, such an algebra $\mathcal B$ is unique, up to equivalence. It is
called the (full) {\it crossed product\/} of $\AA$ and $G$ and denoted by
$\AA\rtimes G$.

\end{definition}

Now we continue the discussion started in the previous subsection but restrict the class of triples $(\Om,\Si,G)$ under consideration. Namely, we assume that $\Om$ is a compact topological space, metrizable and separable; $\Si$ is the $\si$-algebra of Borel sets; $G$ is a finite or countable group of homeomorphisms of $\Om$. 

Let $C(\Om)$ be the space of continuous complex-valued functions on $\Om$. It is a commutative unital $C^*$-algebra, and the group $G$ acts on it by automorphisms.  This makes it possible to form their crossed product $C(\Om)\rtimes G$. 

Let $M$ be a nonzero $G$-quasi-invariant measure on $\Om$. We are going to assign to it a representation of the algebra $C(\Om)\rtimes G$ acting on the Hilbert space $L^2(\Om,M)$. According to Definition \ref{def2.A}, it suffices to specify a covariant representation $(\T_1,\T_2)$, and this is done in the most natural and simple way.  

Namely, elements $f\in C(\Om)$ act as operators of
multiplication,
$$
\T_1(f)h=fh, \qquad h\in L^2(\Om,M),
$$
and the unitary representation $\T_2$ of the group $G$ is given by the  well-known \emph{Koopman-type construction}, see Mackey \cite[pp. 26, 36]{Mackey}:
\begin{equation}\label{eq2.B}
(\T_2(g)h)(\om)=h(g^{-1}(\om))\phi^{1/2}(\om,g), \qquad h\in L^2(\Om,M), \quad  \om\in\Om, \quad
g\in G,
\end{equation}
where $\phi(\om,g)$ is the $1$-cocycle coming from the Radon--Nikod\'ym derivative:
\begin{equation}\label{eq2.A}
\phi(\om,g):=\frac{{}^g\!M}{M}(\om), \qquad \om\in\Om, \quad
g\in G.
\end{equation}

It is readily checked that  $(\T_1,\T_2)$ is a covariant representation. Therefore, it gives
rise to a representation of $C(\Om)\rtimes G$; let us denote it by $\T[M]$. 

\begin{proposition}\label{prop2.B}
Let, as above, $\Om$ be a compact separable metrizable space, $G$ be a finite or countable group of its homeomorphisms, $M$ be a nonzero $G$-quasi-invariant measure, and $\T[M]$ be the associated representation of the crossed product $C(\Om)\rtimes G$. 

{\rm(i)} The equivalence class of the representation $\T[M]$ depends only on the equivalence class of the measure $M$. Conversely, the equivalence class of $M$ can be recovered from the equivalence class of $\T[M]$.

{\rm(ii)} $M$ is ergodic if and only if $\T[M]$ is irreducible. 
\end{proposition}

\begin{proof}
(i) The direct claim is evident: if $M$ is replaced by an equivalent measure, $f M$, then the operator of multiplication by $f^{-1/2}$ determines an isometric map $L^2(\Om, M)\to L^2(\Om, f M)$, which intertwines $\T[M]$ with $\T[f M]$. Conversely, the restriction of $\T[M]$ to $C(\Om)$ is the representation $\T_1$; it is multiplicity free, and its equivalence class is uniquely characterized by the equivalence class of $M$.

(ii) Denote by $\T[M]'$ the commutant of $\T[M]$ --- the algebra of operators on $L^2(\Om,M)$ commuting with the representation $\T[M]$. Given $A\in\Si$, let $P_A$ denote the operator of multiplication by the indicator function of $A$. It is a projection operator. If $A$ is invariant mod $0$, then, by the very construction of $\T[M]$, we have $P_A\in \T[M]'$. 

Conversely, let $P$ be a projection operator in $\T[M]'$. Because $P$ commutes with $\T_1$, it must be of the form $P_A$. Next, from the relations 
$$
P_A \T_2(g) h=\T_2(g) P_Ah,  \qquad  g\in G, \quad h\in L^2(\Om,M),
$$
it follows that $A$ must be invariant mod $0$. 

In this way we obtain a bijective correspondence between projection operators in the commutant $\T[M]'$ and equivalence classes of invariant mod $0$ subsets $A$. This implies the desired claim. 
\end{proof}

\section{The use of hyperoctahedral group $\SS\wr\Z_2$}\label{sect3}

In this section we specialize the correspondence $M\mapsto \T[M]$  to the case when $\Om:=\{0,1\}^\X$, where $\X$ is a countable or finite set.  We are mainly interested in the case of infinite $\X$, but for technical reasons we need the finite case, too. 

If $\X$ is finite, then $\Om$ is a finite set of cardinality $2^{|\X|}$. If
$\X$ is countable, then we equip $\Om$ with the product topology; then $\Om$ becomes a compact topological space. 

As $G$ we take the group $\SS=\SS(\X)$ of finitary permutations of $\X$ (recall that a permutation of a set is called \emph{finitary} if it moves only finitely many points). Of course, if $\X$ is finite, then $\SS$ is finite and consists of all permutations. 

The tautological action of the group $\SS$ on $\X$ induces, in a natural way,
its action on the space $\Om$ by homeomorphisms.

Next, we introduce the \emph{hyperoctahedral group} $\SS\wr\Z_2$, the wreath product of $\SS$ with $\Z_2$. Equivalently, it is the semidirect
product of $\SS$ with the abelian group $\mathcal E$ generated by elements
$\epsi_x$ (where $x$ ranges over $\X$), subject to the relations
$$
\epsi_x^2=e, \qquad \epsi_x\epsi_y=\epsi_y\epsi_x, \qquad x,y\in\X;
$$
the group $\SS$ acts on $\mathcal E$ by
$$
g(\epsi_x)=\epsi_{g(x)}, \qquad g\in\SS, \quad x\in\X,
$$
so that $g\epsi_x g^{-1}=\epsi_{g(x)}$ inside $\SS\wr\Z_2$.

For any finite or countable group $G$, let $\C[G]$ be its group algebra and $C^*[G]$ denote the $C^*$-envelope of $\C[G]$.  Of course, the two algebras are the same for finite $G$. 

\begin{proposition}\label{prop3.A}
Let, as above, $\Om=\{0,1\}^\X$. The crossed product algebra $C(\Om)\rtimes\SS$ is isomorphic to $C^*[\SS\wr\Z_2]$, the $C^*$-algebra of the hyperoctahedral group.
\end{proposition}

The exact form of the isomorphism is indicated in the proof. 

\begin{proof}
Introduce functions $d_x\in C(\Om)$ indexed by points $x\in\X$:
$$
d_x(\om)=1-2\om(x), \qquad \om\in\Om.
$$
In other words, $d_x(\om)$ equals $1$ or $-1$ depending on whether $\om(x)$ equals $0$ or $1$. 

The correspondence $\epsi_x\mapsto d_x$ extends to an embedding of the group
$\mathcal E$ into the group of unitary elements of the algebra $C(\Om)$, which
leads to a morphism $C^*[\mathcal E]\to C(\Om)$. We claim that it is an
isomorphism.

Indeed, consider first the case when $\X$ is finite; then the claim simply
follows from the fact that $C(\Om)$ coincides with the linear span of the group
generated by the elements $d_x$ and those elements are linearly independent. To
handle the case of countable $\X$, we choose an ascending chain of finite subsets
$\X_n$ exhausting $\X$.

The fact that $C^*[\mathcal E]\to C(\Om)$ is an isomorphism is the key
observation, the remainder of the argument being routine. Indeed, this
isomorphism allows us to identify $C(\Om)$ with $C^*[\mathcal E]$, which is a
subalgebra of $C^*[\SS\wr\Z_2]$. Next, $\SS$ is also contained in $C^*[\SS\wr\Z_2]$,
and $C(\Om)$ and $\SS$ together generate $C^*[\SS\wr\Z_2]$. It remains to prove that
every covariant representation $(\T_1,\T_2)$ of $(C(\Om),\SS)$ is factored
through $C^*[\SS\wr\Z_2]$.

To do this, observe that the restriction of $\T_1$ to $\mathcal E$ produces a
unitary representation of this group. Moreover, the covariance property ensures
the commutation relation
$$
\T_2(g)\T_1(\epsi_x)\T_2(g^{-1})=\T_1(\epsi_{g(x)}), \qquad g\in\SS, \quad x\in\X,
$$
which just means that $\T_1\big|_{\mathcal E}$ and $\T_2$ are glued to a
unitary representation of $\SS\wr\Z_2$, which in turn is the same as a
representation of $C^*[\SS\wr\Z_2]$.
\end{proof}

By virtue of the isomorphism between $C(\Om)\rtimes \SS$ and $C^*[\SS\wr\Z_2]$, every representation of the form $\T[M]$ (see Proposition \ref{prop2.B}) can be viewed as a representation of $C^*[\SS\wr\Z_2]$.

The next proposition states that all such representations factor through a proper
quotient of the algebra $C^*[\SS\wr\Z_2]$.

Denote by $I$ the closed two-sided ideal in $C^*[\SS\wr\Z_2]$ generated by the
elements of the form
\begin{equation}\label{eq3.A}
(1-s_{x,y})(1-\epsi_x)(1-\epsi_y), \quad (1-s_{x,y})(1+\epsi_x)(1+\epsi_y),
\end{equation}
where $x$ and $y$ are arbitrary distinct points of $\X$ and
$s_{x,y}\in\SS$ is the corresponding transposition (that is, it switches
$x$ with $y$ and leaves invariant all points from $\X\setminus\{x,y\}$).

\begin{proposition}\label{prop3.B}
For any $\SS$-quasi-invariant measure $M$, the associated representation $\T[M]$ is trivial
on the ideal $I$.
\end{proposition}

\begin{proof}
Given $x\in\Om$, the space $\Om$ can be written as the disjoint union of two
subsets, $\Om_x^1\sqcup\Om_x^0$, where
$$
\Om_x^1=\{\om\in\Om: \om(x)=1\}, \quad \Om_x^0=\{\om\in\Om: \om(x)=0\}.
$$

The key observation is that the restriction of the Radon-Nikod\'ym cocycle
$\phi(\,\cdot\,,s_{x,y})$ onto $\Om_x^1\cap\Om_y^1$ or $\Om_x^0\cap\Om_y^0$
is identically equal to 1, because the transposition $s_{x,y}$ acts trivially
on these subsets. It follows that if a function $h \in L^2(\Om,M)$ is
supported by $\Om_x^1\cap\Om_y^1$ or by $\Om_x^0\cap\Om_y^0$, then
$\T[M](s_{x,y})h=h$, so that $\T[M](1-s_{x,y})h=0$.

On the other hand, by the very definition, the operator $\frac12 \T[M](1-\epsi_x)$  is
the projection onto the subspace of functions supported by $\Om_x^1$.
Likewise, the operator $\frac12\T[M](1+\epsi_x)$ is the projection onto the
complementary subspace formed by functions supported by $\Om_x^0$. Therefore,
the operators
$$
\textrm{$\tfrac14\T[M]((1-\epsi_x)(1-\epsi_y))$ and
$\tfrac14 \T[M]((1+\epsi_x)(1+\epsi_y))$}
$$
are the projections onto the subspaces of functions supported by the subsets
$\Om_x^1\cap\Om_y^1$ and $\Om_x^0\cap\Om_y^0$, respectively. As mentioned above, on
these two subspaces, the operator $\T[M](1-s_{x,y})$ vanishes. This completes the proof.
\end{proof}

\section{From hyperoctahedral group to algebra $\A^0$}\label{sect4}

\subsection{Isomorphism $C^*[\SS\wr\Z_2]/I\to\A^0$}

Note that the generators \eqref{eq3.A} of the ideal $I\subset C^*[\SS\wr\Z_2]$ are
selfadjoint, so that the quotient $C^*[\SS\wr\Z_2]/I$ is a $C^*$-algebra.

The next construction depends on the choice of a linear order on the set $\X$
such that all intervals are finite. For finite $\X$, the latter requirement holds automatically, and for infinite $\X$ it means that the ordered set $(\X,<)$ is
isomorphic to one of the three ordered sets $\Z_{>0}, \Z_{<0}, \Z$. Actually,
reversing the order is unessential for us, so that there are only two
essentially distinct model examples, $\Z_{>0}$ and $\Z$.

Recall that the group $\SS\wr\Z_2$ is generated by the elements $\epsi_x$ and $s_{x,y}$, where $x,y\in\X$, $x\ne y$, and the definition of the algebra $\A^0$ was given in Section \ref{sect1.2}.

\begin{theorem}\label{thm4.A}
We fix a linear order on $\X$ with finite intervals and introduce the following elements of\/ $\A^0$ indexed by the elements $x\in\X$ and by the pairs $x,y\in\X$ such that $x<y${\rm:}
\begin{equation}\label{eq4.B}
\eta_x:=1-2a^+_xa^-_x,  \qquad \eta_{(x,y)}:=\prod_{z:\, x<z<y}\eta_z.
\end{equation}

{\rm(i)}  There exists a surjective morphism of $C^*$-algebras $p:C^*[\SS\wr\Z_2]\to\A^0$, uniquely determined by the correspondence
\begin{gather}
p(\epsi_x):=\eta_x,     \label{eq4.C}\\
p(s_{x,y}):=\frac{1+\eta_x\eta_y}2+\frac{1+\eta_x\eta_y+(1-\eta_x\eta_y)\eta_{(x,y)}}2(a^+_xa_y^-+a^+_ya_x^-),
\quad x<y. \label{eq4.D}
\end{gather}

{\rm(ii)} The kernel of $p$ equals $I$, so that $p$ determines an
isomorphism $C^*[\SS\wr\Z_2]/I\to\A^0$.
\end{theorem}

We first verify the claim of the theorem in the case when $|\X|<\infty$. This constitutes the main part of the proof.  Then we extend the result to the case $|\X|=\infty$. 

\begin{proposition}\label{prop4.A}
The claim of the theorem holds true when $\X$ is finite. 
\end{proposition}

\begin{proof}
Let $N:=|\X|$ and $[N]:=\{1,2,\dots,N\}$. We may assume that $(\X,<)=([N] ,<)$, where the order in $[N]$ is the conventional one.  The space $\Om$ has cardinality $2^N$, and its elements are arbitrary subsets $\om\subseteq[N]$. The $C^*$-algebra of the finite group $\SS\wr\Z_2$ coincides with its group algebra $\C[\SS\wr\Z_2]$. 

Consider the exterior algebra
$\bigwedge\C^N$. It has a distinguished basis $\{e_\om\}$, indexed by arbitrary
subsets $\om\subseteq[N]$: by definition, $e_\om$ is the polyvector
$e_{x_1}\wedge\dots\wedge e_{x_n}$, where $x_1,\dots,x_n$ are  the
elements of $\om$ written in the descending order $x_1>\dots>x_n$. 

The plan of the proof in the finite case is the following. We will deal with two representations in the
same space $\bigwedge\C^N$. One is a representation $\T$ of the group algebra
$\C[\SS\wr\Z_2]$; we show that the kernel of $\T$ equals $I$, so that it is in fact a faithful 
representation of the quotient algebra $C^*[\SS\wr\Z_2]/I$. The other is a faithful
representation $\mathscr F$ of the algebra $\A^0$. We prove that the two representations are consistent with $p$:
\begin{equation}\label{eq4.E}
\T(\epsi_x)=\mathscr F(p(\epsi_x)), \qquad \T(s_{x,y})=\mathscr F(p(s_{x,y})).
\end{equation}
We know that the elements $\epsi_x$ and the elements $s_{x,y}$ together
generate the algebra $\C[\SS\wr\Z_2]$. On the other hand, we also prove that their
images, $p(\epsi_x)$ and $p(s_{x,y})$,  generate the algebra
$\A^0$. This will imply the claim of the proposition.

We proceed to the realization of this plan.

\smallskip

\emph{Step} 1 (Definition of representation $\T$). By definition, the operators $\T(\epsi_x)$ and $\T(s_{x,y})$ act on the basis vectors as follows:
$$
\T(\epsi_x)e_\om=\begin{cases} -e_\om, & x\in \om, \\ e_\om, & x\notin \om,
\end{cases}
$$
and
$$
\T(s_{x,y})e_\om=e_{s_{x,y}(\om)}.
$$
Obviously, they define a representation of the hyperoctahedral group $\SS\wr\Z_2$ and hence of its group algebra. 

Note that the representation $\T$ is in fact the representation $\T[M]$ associated with the counting measure $M$ on $[N]$. More precisely, this holds true under the isomorphism $C(\Om)\rtimes \SS\leftrightarrow  \C[\SS\wr\Z_2]$ and the  isomorphism $\ell^2(\{0,1\}^{[N]})\to\wedge\C^N$ that assigns to each delta function $\de_\om\in\ell^2(\{0,1\}^{[N]})$ the polyvector $e_\om$.  

\smallskip

\emph{Step} 2 (Definition of representation $\mathscr F$). The operators $\mathscr F(a^+_x)$ and $\mathscr F(a^-_x)$ on $\wedge\C^N$ are defined in a standard way. Namely, we set
$$
i(x,\om):=|\{y\in\om: y>x\}|
$$
and then define
$$
\mathscr F(a^+_x)e_\om=e_x\wedge e_\om=\begin{cases}(-1)^{i(x,\om)}e_{\om\cup x}, & x\notin \om,\\
0, & x\in \om, \end{cases}
$$
and
$$
\mathscr F(a^-_x)e_\om=\begin{cases}(-1)^{i(x,\om)}e_{\om\setminus x}, & x\in \om,\\
0, & x\notin \om. \end{cases}
$$

\smallskip

{\it Step\/} 3 (Verification of the first equality in \eqref{eq4.E}). The above
formulas for $\FF$ imply that 
$$
\FF(a^+_xa^-_x)e_\om=\begin{cases} e_\om, & x\in \om, \\ 0, & x\notin \om, \end{cases}.
$$
It follows that
\begin{equation}\label{eq4.F}
\mathscr F(\eta_x)e_\om=\mathscr F(1-2a^+_xa^-_x)e_\om=\begin{cases} -e_\om, & x\in \om, \\ e_\om, & x\notin \om,
\end{cases}
\end{equation}
which coincides with $\T(\epsi_x)e_\om$. Thus,  $\T(\epsi_x)=\mathscr F(\eta_x)$, as desired.

\smallskip

{\it Step\/} 4 (Verification of the second equality in \eqref{eq4.E}). Fix $x<y$
in $[N]$ and set $\om':=s_{x,y}(\om)$. Since $\T(s_{x,y})e_\om=e_{\om'}$, we have to prove
that
$$
\mathscr F(p(s_{x,y}))e_\om=e_{\om'}, \qquad \om\subseteq[N].
$$
Let us examine the following four possible cases (below we use the definitions \eqref{eq4.B}, \eqref{eq4.C}, and \eqref{eq4.D}). 

(i) $\om$ contains neither $x$ nor $y$. Then $\om'=\om$. On the
other hand, both $\mathscr F(a^-_y)$ and $\mathscr F(a^-_x)$ annihilate $e_\om$. Therefore, $\mathscr F(a^+_xa^-_y+a^+_ya^-_x)e_\om=0$. From \eqref{eq4.D} we obtain
$$
\mathscr F(p(s_{x,y}))e_\om=\frac12 \mathscr F(1+\eta_x\eta_y)e_\om=e_\om,
$$
where the second equality holds because both $\mathscr F(\eta_x)$ and $\mathscr F(\eta_y)$ leave $e_\om$ invariant by virtue of \eqref{eq4.F}.

(ii) $\om$ contains both $x$ and $y$. Then again $\om'=\om$ and both
$\mathscr F(a^+_y)$ and $\mathscr F(a^+_x)$ annihilate $e_\om$. Therefore, $\mathscr F(a^+_xa^-_y+a^+_ya^-_x)e_\om=-\mathscr F(a^-_ya^+_x+a^-_xa^+_y)e_\om=0$. From \eqref{eq4.D} we obtain
$$
\mathscr F(p(s_{x,y})e_\om=\frac12\mathscr F(1+\eta_x\eta_y)e_\om=e_\om,
$$
where the second equality holds because both $\mathscr F(\eta_x)$ and $\mathscr F(\eta_y)$ multiply $e_\om$ by $-1$, see \eqref{eq4.F}.

(iii) $\om$ contains $x$ but not $y$. 
Then $\om':=\om\setminus\{x\}\cup\{y\}$. We have
$$
\mathscr F(a^+_xa^-_y+a^+_ya^-_x)e_\om=\mathscr F(a^+_ya^-_x)e_\om=(-1)^{\ell(x,y)}e_{\om'},
$$
where
$$
\ell(x,y):=|\{z\in \om: x<z<y\}|.
$$

Next, $\mathscr F(\eta_x\eta_y)e_\om=-e_\om$, so that $\mathscr F(1+\eta_x\eta_y)e_\om=0$. Likewise,
$\mathscr F(\eta_x\eta_y)e_{\om'}=-e_{\om'}$ and $\mathscr F(1+\eta_x\eta_y)e_{\om'}=0$. Therefore, we obtain from \eqref{eq4.D}
$$
\mathscr F(p(s_{x,y})e_\om=(-1)^{\ell(x,y)}\frac12\mathscr F(1-\eta_x\eta_y)\FF(\eta_{(x,y)})e_{\om'}.
$$
Observe now that $\mathscr F(\eta_{(x,y)})e_{\om'}=(-1)^{\ell(x,y)}e_{\om'}$. This finally gives
$$
\mathscr F(p(s_{x,y})e_\om=\frac12\mathscr F(1-\eta_x\eta_y)e_{\om'}=e_{\om'},
$$
as desired.

(iv) $\om$ contains $y$ but not $x$. Then
$\om':=\om\setminus\{y\}\cup\{x\}$. Just the same argument as in (iii)
shows that $\mathscr F(p(s_{x,y}))e_\om=e_{\om'}$.

\smallskip

\emph{Step} 5 (Surjectivity of $p$). The above steps show that $p$
extends uniquely to an algebra morphism $\C[\SS\wr\Z_2]\to\A^0$. We are going to
prove that its image, which we denote by $\operatorname{Im} p$, is the whole algebra $\A^0$. It suffices to check
that for any pair $x<y$,  both $a^+_xa^-_y$ and $a^+_ya^-_x$ are contained in $\operatorname{Im} p$.

Obviously, $\operatorname{Im} p$ contains the elements $\eta_z$ for arbitrary
$z\in[N]$. It also contains the element $p(s_{x,y})$. Therefore,
denoting
$$
\zeta_{x,y}:=\frac{1+\eta_x\eta_y+(1-\eta_x\eta_y)\eta_{(x,y)}}2,
$$
we see that $\operatorname{Im} p$ contains the element $\zeta_{x,y}(a^+_xa^-_y+a^+_ya^-_x)$.

Next, observe that $\zeta_{x,y}^2=1$. Indeed, this is directly checked using the fact that the elements $\eta_x$, $\eta_y$, and $\eta_{(x,y)}$ pairwise commute and their squares are equal to $1$. 

Thus, $\operatorname{Im} p$ contains  the element $a^+_xa^-_y+a^+_ya^-_x$. Observe that $\operatorname{Im} p$ also contains all elements of the form $a^+_za^-_z$, $z\in[N]$. Now, multiplying
$a^+_xa^-_y+a^+_ya^-_x$ on the left by $a^+_xa^-_x$ we get $a^+_xa^-_y$. Likewise,
multiplication on the left by $a^+_ya^-_y$ extracts $a^+_ya^-_x$.

\smallskip

\emph{Step} 6 (The kernel of $p$). Let $\Ker p$ denote the kernel of $p$. We are going to prove that  $\Ker p=I$. 

As noted above, $\T$ is essentially the canonical representation associated with the counting measure on $[N]$. It follows that $\Ker p$ contains $I$, by virtue of Proposition \ref{prop3.A}. Of course, this can also be verified directly from the definition of $\T$. The point is that $\Ker p$ cannot be strictly greater than $I$, which is not so evident. 

Since $\C[\SS\wr\Z_2]$ is a finite-dimensional semisimple algebra, every two-sided
ideal is uniquely characterized by the set of irreducible representations that
are trivial on it. We will show  that such irreducible representations are the
same for $\Ker p$ and for $I$; then the desired equality $\Ker p=I$ will follow.

The representation $\mathscr F$ of the algebra $\A^0$ is faithful and 
decomposes into the multiplicity free direct sum of $N+1$ irreducible
representations realized in the homogeneous components
$\bigwedge^m\C^N\subset\bigwedge\C^N$, $0\le m\le N$. Let us denote these
representations as $\mathscr F_m$. Then the representations $\mathscr F_m\circ p$ are precisely those
irreducible representations of $\C[\SS\wr\Z_2]$ that kill the ideal $\Ker p$.

On the other hand, the representations of the algebra $\C[\SS\wr\Z_2]$ are the same
as the representations of the group $\SS\wr\Z_2$. Recall that this group is the semidirect
product of $\SS_N$, the finite symmetric group of degree $N$, and the
commutative group $\mathcal E_N:=\Z_2^N$. The irreducible representations of
the semidirect product $\SS_N\ltimes \mathcal \Z_2^N$ are well known; recall their
description.

They are parameterized by the pairs $(\rho,\chi)$, where $\chi$ is a character of
$\Z_2^N$ (we need only to pick a representative in each $\SS_N$-orbit in
the dual to $\Z_2^N$) and $\rho$ is an irreducible representation of
$\SS^\chi_N$, the stabilizer of $\chi$ in $\SS_N$. The corresponding
representation $\pi_{\rho,\chi}$ of $\SS_N\ltimes\Z_2^N$ is induced by
the representation $\rho\otimes \chi$ of the subgroup
$\SS^\chi_N\ltimes\Z_2^N$. The $\SS_N$-orbits in the dual to $\Z_2^N$ are parameterized by numbers $n=0,\dots,N$; as a representative of the
$n$th orbit we pick the following character
$$
\chi_m(\epsi_x)=\begin{cases}-1, & x=1,\dots,m\\
1, & x=m+1,\dots,N.\end{cases}
$$
Its stabilizer is the Young subgroup $\SS_m\times \SS_{N-m}\subseteq \SS_N$, and an
irreducible representation of this subgroup is written as the tensor product
$\rho=\rho'\otimes\rho''$, where $\rho'$ and $\rho''$ are irreducible
representations of $\SS_m$ and $\SS_{N-m}$, respectively.

From this picture it is easily seen that the irreducible representations
$\mathscr F_m\circ p$ are precisely those $\pi_{\rho,\chi}$'s for which
$\rho$ is trivial. Let us call them \emph{elementary representations\/}. Now the
problem reduces to the following one: check that a non-elementary irreducible
representation  cannot kill the ideal $I$.

Let $\pi_{\rho,\chi_m}$ be a non-elementary representation. Then at least one
of representations $\rho'$ and $\rho''$ is nontrivial. Assume $\rho'$ is
nontrivial; then $m\ge2$. We claim that then the element $(1-s_{1,2})(1-\epsi_1)(1-\epsi_2)$
(which is one of the generators of the ideal $I$, see \eqref{eq3.A}) acts nontrivially. 

Indeed, it acts nontrivially already in the inducing representation $\rho\otimes\chi_n$,
because in that representation, $\epsi_1$ and $\epsi_2$ acts as the scalar $-1$, so that
$(1-\epsi_1)(1-\epsi_2)$ acts as the scalar 4; on the other hand, the action of $1-s_{1,2}$
is nontrivial since $\rho'$ is nontrivial. The same argument works in the case
when $\rho''$ is nontrivial; then $N-n\ge2$ and we use the element
$(1-s_{m+1,m+2})(1+\epsi_{m+1})(1+\epsi_{m+2})\in I$.

\smallskip

This completes the proof.
\end{proof}

\smallskip

\begin{proof}[Proof of Theorem \ref{thm4.A}]
We suppose now that $(\X,<)$ is a countable, linearly ordered set with finite intervals. The extension to this case is straightforward. Indeed, we realize
$\X$ as the union of an ascending chain of finite intervals $\X_N$ and write
every algebra in question as the closure of  the ascending chain of the
corresponding finite-dimensional subalgebras associated with the intervals
$\X_N$. The key fact is that the map $p$ is ``local'', so that
it is consistent with the embedding of the $N$th subalgebra into the $(N+1)$th
one.

This completes the proof of the theorem.
\end{proof}

\subsection{Representations of algebra $\A^0$ associated with quasi-invariant measures}

From Theorem \ref{thm4.A} we immediately obtain

\begin{corollary}\label{cor4.A}
Let, as above, $(\X,<)$ be a finite or countable linearly ordered set {\rm(}with finite intervals in the infinite case{\rm);} $\Om=\{0,1\}^\X${\rm;} $\SS$ be the group of finitary permutations of $\X${\rm;} $\A^0$ be the gauge invariant subalgebra of the $\CAR$ algebra $\A$.

For any $\SS$-quasi-invariant $\si$-finite measure $M$ on $\Om$, there exists a representation $T[M]$ of the algebra $\A^0$ on the space $L^2(\Om,M)$, uniquely determined by the property
$$
\T[M](a)=T[M](p(a)), \qquad \forall a\in C^*[\SS\wr\Z_2],
$$
where $p:C^*[\SS\wr\Z_2]\to \A^0$ is the homomorphism established in Theorem \ref{thm4.A}.
\end{corollary}

Here is an alternative characterization of the representation $T[M]$.

\begin{proposition}\label{prop4.C}
In the assumptions of Corollary \ref{cor4.A}, the representation $T=T[M]$ of the algebra $\A^0$ is uniquely determined through the representation $\T=\T[M]$ of the algebra $C^*[\SS\wr\Z_2]\simeq C(\Om)\rtimes\SS$ by the following two conditions: first,
\begin{equation}\label{eq4.H1}
T(a^+_xa^-_x)=\tfrac12(1-\T(\epsi_x)), \qquad x\in\X;
\end{equation}
and, second,
\begin{equation}\label{eq4.H2}
T(a^+_xa^-_y+a^+_ya^-_x)=\T(\wt\epsi_{x,y})\left(\T(s_{x,y})-\T\left(\dfrac{1+\epsi_x\epsi_y}2\right)\right), \quad x, y\in\X, \; x\ne y,
\end{equation}
where we use the following notation:
\begin{equation}\label{eq4.H3}
\wt\epsi_{x,y}:=\frac{1+\epsi_x\epsi_y+(1-\epsi_x\epsi_y)\epsi_{(x,y)}}2,
\end{equation}
\begin{equation}\label{eq4.H4}
\epsi_{(x,y)}:=\prod_{z\in(x,y)}\epsi_z, \qquad (x,y):=\{z: \min(x,y)<z<\max(x,y)\}.
\end{equation}
\end{proposition}

\begin{proof}
This is simply a reformulation of Corollary \ref{cor4.A} using the fact that $\wt\epsi^2_{x,y}=1$.
\end{proof}

\begin{remark}\label{rem4.C}
The second condition may be reformulated as follows. Given two distinct points $x,y\in\X$, consider  the partition
$\Om=\Om^+_{x,y}\sqcup \Om^-_{x,y}$. where
$$
\Om^+_{x,y}:=\{\om\in\Om: \om(x)=\om(y)\}, \quad \Om^-_{x,y}:=\{\om\in\Om: \om(x)\ne\om(y)\}.
$$
Then we have the direct sum decomposition
$$
L^2(\Om,M)=L^2(\Om^+_{x,y},M)\oplus L^2(\Om^-_{x,y}, M).
$$
The idea is that the operator on the right-hand of \eqref{eq4.H2} takes a much simpler form when written separately on the subspaces  $L^2(\Om^\pm_{x,y},M)$. Indeed, $\T(s_{x,y})$ preserves $L^2(\Om^-_{x,y}, M)$ and acts identically on $L^2(\Om^+_{x,y}, M)$, while $\T(\epsi_x\epsi_y)$ acts on $L^2(\Om^\pm_{x,y},M)$ as $\pm1$. It follows that \eqref{eq4.H2} is equivalent to the combination of two relations:
\begin{equation}\label{eq4.I1}
T(a^+_xa^-_y+a^+_ya^-_x)\big|_{L^2(\Om^+_{x,y},M)}=0,
\end{equation}
\begin{equation}\label{eq4.I2}
T(a^+_xa^-_y+a^+_ya^-_x)\big|_{L^2(\Om^-_{x,y},M)}=\prod_{z\in(x,y)}\T(\epsi_z)\cdot\T(s_{x,y}).
\end{equation}

\end{remark}

\begin{remark}\label{rem4.A}
Evidently, claims (i) and (ii) of Proposition \ref{prop2.B} hold true for  representations $T[M]$ of the algebra $\A^0$.
\end{remark}

\begin{remark}\label{rem4.B}
Our construction of representations $T[M]$ of the algebra $\A^0$, associated to $\SS$-quasi-invariant measures $M$, is similar to the constructions described in the book Str\u{a}til\u{a}--Voiculescu \cite[Chapter IV]{SV-1975} (see also their paper \cite{SV-1978}) and in the expository paper Vershik--Kerov \cite[Section 6]{VK}. But there are technical differences: Str\u{a}til\u{a} and Voiculescu use a different symmetry group, while Vershik and Kerov make accent on the fact that instead of group actions one may  deal with equivalence relations. However, for our purposes, a simple reference to these works would not be enough, and we had to present the necessary material in our own way.
\end{remark}

\subsection{The case of counting measure}

In this subsection we consider an arbitrary countable $\SS$-invariant subset $\Theta\subset\Om$ (for instance, the orbit of an arbitrary $\om\in\Om$) and take as $M$  the counting  measure on $\Theta$. According to Corollary \ref{cor4.A}, such a measure gives rise to a representation $T[M]$ of the algebra $\A^0$, which acts on the Hilbert space  $\ell^2(\Theta)$. 
In Proposition \ref{prop4.B} we describe $T[M]$ explicitly, using the natural orthonormal basis $\{\de_\om:\om\in\Theta\}$ formed by the delta functions. This result is then used in Section \ref{sect5}.

To formulate the proposition, we need to introduce some notation. 

For any two distinct points $x,y\in\X$ we set
$$
\sgn(x,y):=\begin{cases}1, & x>y, \\ -1, & x<y.\end{cases}
$$

Next, let $(x_1,\dots,x_n)$ and $(y_1,\dots,y_n)$ be two ordered $n$-tuples of points of $\X$ such that $x_i\ne x_j$ and $y_i\ne y_j$ for $i\ne j$. Given $\om\in\Om$, we set
\begin{multline}\label{eq4.A}
\sgn(x_1,\dots,x_n;y_1,\dots,y_n; \om)\\
:=\begin{cases} \prod\limits_{u\in\om\setminus\{y_1,\dots,y_n\}}\;\prod\limits_{i=1}^n\sgn(x_i,u)\sgn(y_i,u)
\cdot \prod\limits_{1\le i<j\le n}\sgn(x_i,x_j)\sgn(y_i,y_j),\\
\text{if $\om\supseteq\{y_1,\dots,y_n\}$ and $(\om\setminus\{y_1,\dots,y_n\})\cap\{x_1,\dots,x_n\}=\varnothing$};\\
0, \; \text{otherwise}.
\end{cases}
\end{multline}
Note that $\om$ may be infinite, but the product over $u$ is in fact always finite, because for any fixed $x$ and $y$, the product $\sgn(x,u)\sgn(y,u)$ takes the value $-1$ for finitely many $u$ only.

\begin{proposition}\label{prop4.B}
In this notation, we have for any $\om\in\Theta$
\begin{equation}\label{eq4.G}
T(a^+_{x_n}\dots a^+_{x_1}a^-_{y_1}\dots a^-_{y_n})\de_\om=\sgn(x_1,\dots,x_n;y_1,\dots,y_n; \om) \de_{(\om\setminus\{y_1,\dots,y_n\})\cup\{x_1,\dots,x_n\}}.
\end{equation}
\end{proposition}

\begin{proof}
As in the proof of Theorem \ref{thm4.A}, the case of infinite $\X$ is reduced to the case of finite $\X$. Then one may assume, without loss of generality, that $\Theta$ is a single $\SS$-orbit. Then we observe that $\ell^2(\Theta)$ is a subspace of $\ell^2 (\Om)$ (both spaces are finite dimensional!) and finally we may assume that $\Theta=\Om$. 

Thus, we may assume that $(\X,<)=([N],<)$ for some $N$.  This allows us to pass to the representation $\mathscr F$ on the space $\bigwedge\C^N$, using the isomorphism $\ell^2(\Om)\to \bigwedge\C^N$, as in the proof of Proposition \ref{prop4.A}.

Examine the two cases, as in \eqref{eq4.A}. In the second case, both sides of \eqref{eq4.G} vanish: for the right-hand side this holds by the very definition, and for the left-hand this is seen from the definition of $\mathscr F$. Thus, it remains to examine the first case. Then it is clear that we only have to check that the sign on the right is correct.

Observe that both sides of \eqref{eq4.G} are skew-symmetric with respect to separate permutations of $x_i$'s and $y_i$'s. Indeed, for the left-hand side this follows from the defining relations, and for the right-hand side this is seen from \eqref{eq4.A}. This makes it possible to assume that $x_1>\dots>x_n$ and $y_1>\dots>y_n$. 

Finally, we use the equality 
$$
a^+_{x_n}\dots a^+_{x_1}a^-_{y_1}\dots a^-_{y_n}=(a^+_{x_1}a^-_{y_1})\dots(a^+_{x_n}a^-_{y_n})
$$
and write the basis polyvector $e_\om\in\bigwedge\C^N$ corresponding to $\de_\om\in\ell^2(\Om)$ as $e_{i_1}\wedge\dots\wedge e_{i_k}$. We know that all the $y_j$'s are contained among the indices $i_1,\dots,i_k$. Then from the definition of $\mathscr F$ it is seen that
$$
\mathscr F((a^+_{x_1}a^-_{y_1})\dots(a^+_{x_n}\,a^-_{y_n}))e_{i_1}\wedge\dots\wedge e_{i_k}
$$
is the polyvector obtained from $e_{i_1}\wedge\dots\wedge e_{i_k}$ by replacing each $e_{y_j}$ with $e_{x_j}$. Reordering the resulting polyvector will produce the sign $(-1)^m$, where $m$ is equal to the number of ``inversions'', that is, $m$ is the number of pairs of indices $(u, r)$ such that $u\in\om\setminus\{y_1,\dots,y_n\}$, $r=1,\dots,n$, and the differences $u-y_r$ and $u-x_r$ have opposite sign. From this it follows that $(-1)^m$ is precisely the desired sign.
\end{proof}

\subsection{States $\tau[M]$ and perfect measures}

Under the additional assumption that $M$ is a probability measure we can associate with $M$ a state on $\A^0$. 
\begin{definition}\label{def4.A}
Let $M$ be an $\SS$-quasi-invariant probability measure on $\Om$ and  $\one$ stand for the constant function on $\Om$ equal to $1$, which we regard as a vector of the Hilbert space $L^2(\Om,M)$. We denote by $\tau[M]$ the corresponding state on $\A^0$:
\begin{equation}\label{eq4.H}
\tau[M](a):=(T[M](a)\one,\one), \qquad a\in\A^0.
\end{equation}
\end{definition}

By the very definition of $T[M]$, the restriction of $\tau[M]$ to the subalgebra $C(\Om)\subset\A^0$ is the state $\EE_M$ corresponding to $M$. Note that $\one$ is a cyclic vector of $T[M]$, because it is already a cyclic vector with respect to the action of the subalgebra $C(\Om)\subset\A^0$. Therefore, the  whole information about the representation $T[M]$ is encoded in $\tau[M]$.

\begin{definition}\label{def4.B}
Let $M$ be an $\SS$-quasi-invariant probability measure on $\Om$. We say that $M$ is a \emph{perfect measure} if the state $\tau[M]$ is a quasifree state $\varphi[K]$ for some positive contraction $K$ on $\ell^2(\X)$ (see Definition \ref{def1.B}). Further, the kernel $K(x,y):=(K e_y,e_x)$  will be called the  \emph{canonical correlation kernel} of $M$. 
\end{definition}

It is not evident a priori that perfect measures exist, but concrete examples will be given in Sections \ref{sect5} and \ref{sect8}. 

\begin{remark}\label{rem4.D}
The product measures are not perfect measures. Here is a sketch of proof. A product measure $M$ on $\Om=\{0,1\}^\X$ is determined by a collection of numbers $\{p_x:x\in\X\}$ such that $0<p_x<1$; it is a determinantal measure with the diagonal correlation kernel $K_{\operatorname{diag}}(x,y):=\de_{xy}p_x$. Obviously, $M$  is $\SS$-quasi-invariant, so that the state $\tau[M]$ is well defined.  A direct computation shows that 
$$
\tau[M](a^+_x a^-_y)=(p_x(1-p_x)p_y(1-p_y))^{1/2}, \quad x\ne y.
$$
Therefore, the only choice for the canonical kernel would be
$$
K(x,x)=p_x, \qquad K(x,y)=(p_x(1-p_x)p_y(1-p_y))^{1/2}, \quad x\ne y.
$$ 
But $\tau[M]\ne\varphi[K]$, because, for $x_1\ne x_2$ one has
$$
\tau[M](a^+_{x_2}a^+_{x_1}a^-_{x_1}a^-_{x_2})=\tau[M]((a^+_{x_1}a^-_{x_1})(a^+_{x_2}a^-_{x_2}))=p_{x_1}p_{x_2},
$$
which is distinct from 
$$
\det\begin{bmatrix}K(x_1,x_1) & K(x_1,x_2)\\ K(x_2,x_1) & K(x_2,x_2)\end{bmatrix}=p_{x_1}p_{x_2}-p_{x_1}(1-p_{x_1})p_{x_2}(1-p_{x_2}).
$$
\end{remark}

\begin{remark}\label{rem4.E}
From the definition of the representation $T[M]$ and the state $\tau[M]$ it is seen that $\tau[M]$ takes real values on all monomials composed from the generators $a^+_xa^-_y$ of the algebra $\A^0$. It follows, in particular, that canonical correlation kernels must be real valued. 
\end{remark}

\subsection{Particle/hole involution}\label{sect4.5}
For $\om\in\Om$, set $\om^\circ:=\Om\setminus\om$. Clearly, $(\om^\circ)^\circ=\om$. The map $\om\mapsto\om^\circ$ is a homeomorphism of $\Om$ called the \emph{particle/hole involution}. It induces an involutive transformation $M\mapsto M^\circ$ of measures on the space $\PP(\Om)$. The particle/hole involution commutes with the action of $\SS$ and hence preserves the set of $\SS$-quasi-invariant measures.  In the next proposition we describe the link between $T[M]$ and $T[M^\circ]$. 

To state it we need a new notation. Recall that the ordered set $(\X,<)$ is isomorphic to one of the sets, $\{1,\dots,N\}$, $\Z_{\ge0}$, $\Z_{\le0}$, $\Z$, and denote by $\nu$ the corresponding bijection (in the last case, $\nu$ is defined up to a shift, but this does not matter). The correspondence 
$$
a^\pm_x\mapsto (-1)^{\nu(x)}a^\mp_x, \qquad x\in\X,
$$
extends to an involutive automorphism of the algebra $\A$, which preserves the subalgebra $\A^0$. On $\A^0$, its action is uniquely determined by the correspondence 
\begin{equation}\label{eq4.K}
a^+_xa^-_y\mapsto (-1)^{\nu(x)-\nu(y)}a^-_xa^+_y, \quad x,y\in\X.
\end{equation} 

Next, given a measure $M$ on $\Om$, let $U: L^2(\Om,M)\to L^2(\Om,M^\circ)$ denote the natural isometry of these two Hilbert spaces, induced by the homeomorphism $\om\mapsto\om^\circ$. 

\begin{proposition}\label{prop4.D}
Let  $M$ be an $\SS$-quasi-invariant measure on $\X$. Then
\begin{equation}\label{eq4.J}
U^{-1}\circ T[M^\circ](a)\circ U=T[M](a^\circ), \quad a\in\A^0,
\end{equation}
where $a\mapsto a^\circ$ is the involutive automorphism of\/ $\A^0$ defined by \eqref{eq4.K}. 
\end{proposition}

\begin{proof}
The particle/hole involution gives rise to an automorphism of the algebra $C(\Om)\rtimes\SS$, which is identical on $\SS$ and reduces on $C(\Om)$ to the transformation $f\mapsto f^\circ$, where $f^\circ(\om):=f(\om^\circ)$. Under the identification $C(\Om)\rtimes\SS\to C^*[\SS\wr\Z_2]$, we get an automorphism of the latter algebra, which is still identical on $\SS$ and sends $\epsi_x$ to $-\epsi_x$ for each $x\in\X$. Let us denote that automorphism as $b\mapsto b^\circ$. In this notation, we have 
$$
U^{-1}\circ \T[M^\circ](b)\circ U=\T[M](b^\circ), \quad b\in C^*[\SS\wr\Z_2].
$$

On the other hand, it is seen from \eqref{eq3.A} that the automorphism $b\mapsto b^\circ$ preserves the ideal $I$, the kernel of the projection $p: C^*[\SS\wr\Z_2]\to\A^0$ and hence gives rise to an automorphism of $\A^0$. It suffices to check that the latter automorphism coincides with the automorphism $a\mapsto a^\circ$ defined by \eqref{eq4.K}. This is clear for $a=a^+_xa^-_x$ by virtue of \eqref{eq4.H1}. It remains to examine the case of $a=a^+_xa^-_y+a^+_ya^-_x$, where $x\ne y$. To do that, it is more convenient to use, instead of \eqref{eq4.H2}, the relations \eqref{eq4.I1} and \eqref{eq4.I2}. Then our task reduces to checking the following two relations for $a:=a^+_xa^-_y+a^+_ya^-_x$ (below we abbreviate $T=T[M]$ and $\T=\T[M]$, as in \eqref{eq4.I1} and \eqref{eq4.I2}):
\begin{equation}\label{eq4.L1}
T(a^\circ)\big|_{L^2(\Om^+_{x,y},\, M)}=0
\end{equation}
and
\begin{equation}\label{eq4.L2}
T(a^\circ)\big|_{L^2(\Om^-_{x,y},\, M)}=\prod_{z\in(x,y)}\T((\epsi_z)^\circ)\cdot\T((s_{x,y})^\circ).
\end{equation}
Observe now that $a^\circ=-(-1)^{\nu(x)-\nu(y)}a$ and recall that $(\epsi_z)^\circ=-\epsi_z$ and $(s_{x,y})^\circ=s_{x,y}$. Now it is clear that \eqref{eq4.L1} is equivalent to \eqref{eq4.I1}. Next, \eqref{eq4.L2} is equivalent to \eqref{eq4.I2}, because $(-1)^{|(x,y)|}=-(-1)^{\nu(x)-\nu(y)}$. 
\end{proof}

\begin{proposition}\label{prop4.E}
Let  $M$ be  a perfect measure and $K(x,y)$ be its canonical correlation kernel. Then the measure $M^\circ$ obtained from $M$ by the particle/hole involution  is also perfect and its canonical correlation kernel is 
$$
K^\circ(x,y):=\de_{xy}-(-1)^{\nu(x)-\nu(y)}K(x,y), \qquad x,y\in\X.
$$
\end{proposition}

This result is used below in the proof of Theorem \ref{thm8.B}. Note that the factor $(-1)^{\nu(x)-\nu(y)}$ depends on the linear order in $\X$. This is a manifestation of the fact that the very definition of canonical correlation kernel is tied to a linear order. Below we denote by $K^\circ$ the operator with the matrix $K^\circ(x,y)$. 

\begin{proof}
Since the constant function $\one$ is invariant under the particle/hole involution, we have
$$
\tau[M^\circ](a)=\tau[M](a^\circ), \quad a\in\A^0,
$$
where $a\mapsto a^\circ$ is the automorphism of $\A^0$ defined by \eqref{eq4.K} (here we also use Proposition \ref{prop4.D}).  Therefore, it suffices to prove that 
\begin{equation}\label{eq4.M1}
\varphi[K](a^\circ)=\varphi[K^\circ](a), \quad a\in\A^0.
\end{equation}

Let $a\mapsto a^\lozenge$ denote the automorphism of $\A$ defined by  $(a^\pm_x)^\lozenge=a^\mp_x$ for each $x\in\X$. The automorphism $a\mapsto a^\circ$ of the algebra $\A$ is the composition of $a\mapsto a^\lozenge$ with the automorphism multiplying each $a^\pm_x$ by $(-1)^{\nu(x)}$. Therefore, it suffices to prove the analogue of \eqref{eq4.M1} for $a\mapsto a^\lozenge$, which amounts to 
\begin{equation}\label{eq4.M}
\varphi[K](a^\lozenge)=\varphi[1-K](a), \quad a\in\A^0.
\end{equation}

This fact can be extracted from Araki \cite{Araki}, but for reader's convenience we give a proof. We need a little preparation. Let $V$ denote the (algebraic) linear span of the elements $a^\pm_x$, $x\in\X$. We equip $V$ with the symmetric bilinear form $\langle\ccdot, \ccdot\rangle$ defined by
$$
\langle a^+_x,a^+_y\rangle=\langle a^-_x,a^-_y\rangle=0, \quad \langle a^+_x,a^-_y\rangle=\de_{xy},
$$
so that $V$ becomes a quadratic space. Let $\Cliff(V)$ be the corresponding Clifford algebra; it is a dense subalgebra of $\A$. Let us say that a linear functional $\Phi:\Cliff(V)\to\C$ is a \emph{Pfaffian functional}\/ if $\Phi(1)=1$ and for any $n=1,2,\dots$ and any ordered $n$-tuple $v_1,\dots,v_n\in V$ one has
$$
\Phi(v_1\dots v_n)=\begin{cases}0, & \text{$n$ odd}, \\ \Pf(\mathcal A_N), & \text{$n$ even},\end{cases}
$$
where $\Pf(\mathcal A_N)$ is the Pfaffian of the $n\times n$ skew-symmetric matrix $\mathcal A_N$ with elements
$$
\mathcal A_N(i,i)=0,  \qquad \mathcal A_N(i,j)=-\mathcal A_N(j,i)=\Phi(v_iv_j), \quad i<j.
$$

A Pfaffian functional $\Phi$ is uniquely determined by the associated bilinear form on $V\times V$ given by
$$
\wt \Phi(v_1,v_2):=\Phi(v_1v_2),
$$
which satisfies the relation
$$
\wt\Phi(v_1,v_2)+\wt\Phi(v_2,v_1)=\langle v_1,v_2\rangle, \quad v_1,v_2\in V.
$$
Conversely, any bilinear form satisfying this relation gives rise to a Pfaffian functional: this can be deduced from Wick's formula or proved directly using the definition of $\Cliff(V)$ as a quotient of the tensor algebra of $V$. 

It follows that  if $\Phi(a)$ is a Pfaffian state, then the functional $\Phi^\lozenge: a\mapsto \Phi(a^\lozenge)$ is a Pfaffian functional, too, and the corresponding bilinear form $\wt\Phi^\lozenge(v_1,v_2)$ is given by 
$$
\wt\Phi^\lozenge(v_1,v_2)=\wt\Phi(v^\lozenge_1,v^\lozenge_2).
$$

Now we return to the proof of \eqref{eq4.M}. Let $\Cliff^0(V):=\A^0\cap\Cliff(V)$. The key observation is that the restriction of $\varphi[K]$ to $\Cliff^0(V)$ is given by the Pfaffian functional corresponding to the bilinear form 
$$
\wt\Phi(a^+_x, a^+_y)=\wt\Phi(a^-_x, a^-_y)=0, \quad \wt\Phi(a^+_x, a^-_y)=K(x,y), \quad \wt\Phi(a^-_y, a^+_x)=1-K(x,y).
$$
It follows that the functional $a\mapsto \varphi[K](a^\lozenge)$ is also given by a Pfaffian functional, and the corresponding bilinear form is
$$
\wt\Phi^\lozenge(a^+_x, a^+_y)=\wt\Phi(a^-_x, a^-_y)=0, \quad \wt\Phi^\lozenge(a^+_x, a^-_y)=1-K(y,x), \quad \wt\Phi^\lozenge(a^-_y, a^+_x)=K(y,x).
$$

In general, the kernel $K(x,y)$ is Hermitian symmetric. But since it is assumed to be a canonical kernel, it must be real valued (see Remark \ref{rem4.E}), so that $K(x,y)=K(y,x)$. symmetric. We conclude that the state $a\mapsto \varphi[K](a^\lozenge)$ coincides with $\varphi[1-K]$. 
\end{proof}

\section{Discrete orthogonal polynomial ensembles}\label{sect5}

So far $(\X,<)$ was an abstract set. In this section we deal with a more concrete situation: namely, we suppose that $\X$ is realized as a subset of $\R$. As before, we suppose that $\X$ is finite or countable. Next, we fix a positive integer  $N$ and  a strictly positive function $W(x)$ on $\X$. If $\X$ is finite, we suppose that $N<|\X|$, and if $\X$ is infinite, we require  
$$
\sum_{x\in\X}x^{2N} W(x)<\infty.
$$
Denote by $\Om_N\subset\Om$ the set of $N$-point subsets of $\X$. The \emph{$N$-point orthogonal polynomial ensemble with weight function $W$} is determined by the probability measure $M_{N,W}\in\PP(\Om)$ supported by $\Om_N$ and given by 
\begin{equation}\label{eq5.G}
M_{N,W}(\om)=\frac1{\mathcal Z_N}\cdot\prod_{i=1}^N W(u_i)\cdot \prod_{1\le i<j\le N} (u_i-u_j)^2, \qquad \om\in\Om_N,
\end{equation}
where $u_1,\dots,u_N$ are the points of $\om$ listed in an arbitrary order and $\mathcal Z_N$ is the normalization constant. 

Let $\K_{N,W}$ be the operator of orthogonal projection onto the $N$-dimensional subspace in $\ell^2(\X)$ spanned by the functions $x^n W^{1/2}(x)$, $n=0,1,\dots,N-1$, and let $K_{N,W}(x,y)=(\K_{N,W}e_y,e_x)$ be the corresponding kernel on $\X\times\X$. It is well known  that $M_{N,W}$ is a determinantal measure admitting $K_{N,W}(x,y)$ as a correlation kernel (see e.g. K\"onig  \cite{Konig}).  

We assume that the order on $\X$ is induced by the conventional order on $\R$.

\begin{theorem}\label{thm5.A}
Under the assumptions stated above, assume additionally, in the case $|\X|=\infty$, that the ordered set $(\X,<)$ has finite intervals. Then $M_{N,W}$ is a perfect measure in the sense of Definition \ref{def4.B} and $K_{N,W}(x,y)$ is its canonical kernel. \end{theorem}

\begin{proof}
Note that the set  $\Om_N$ forms an $\SS$-orbit and $M_{N,W}$ charges every $\om\in\Om_N$.  It follows that $M_{N,W}$ is $\SS$-quasi-invariant. Therefore,  Corollary \ref{cor4.A}  is applicable and the state $\tau[M_{N,W}]$ on $\A^0$ is well defined.  Our task is to prove that $\tau[M_{N,W}]$ is the quasifree state $\varphi[\K_{N,W}]$.

\emph{Step} 1. Consider the Hilbert space $\ell^2(\Om_N)$ with its distinguished basis $\{\de_\om: \om\in\Om_N\}$. Let $T_N$ denote the representation of $\A^0$ on that space, associated with the counting measure on $\Om_N$. The operator of multiplication by the function $(M_{N,W}(\ccdot))^{1/2}$ maps isometrically the weighted Hilbert space $\ell^2(\Om_N,M_{N,W})$ onto the Hilbert space $\ell^2(\Om_N)$, intertwines the representations $T[M_{N,W}]$ and $T_N$, and takes the distinguished vector $\one$ to the unit vector
$$
F:=\sum_{\om\in\Om_N} (M_{N,W}(\om))^{1/2} \de_\om 
$$
in $\ell^2(\Om_N)$. It follows that
$$
\tau[M_{N,W}](a)=(T_N(a)F,F), \qquad a\in\A^0,
$$
where the scalar product on the right-hand side is that of $\ell^2(\Om_N)$.

Therefore, the desired equality $\tau[M_{N,W}]=\varphi[\K_{N,W}]$ is reduced to the following one:
\begin{equation}\label{eq5.A}
(T_N(a^+_{x_n}\dots a^+_{x_1}a^-_{y_1}\dots a^-_{y_n})F,F)=\det[K_{N,W}(y_i,x_j)],
\end{equation}
where $n=1,2,\dots,N$, and $\{x_1,\dots,x_n\}$ and $\{y_1,\dots,y_n\}$ are arbitrary $n$-tuples of points from $\X$. Without loss of generality we may assume that $x_i\ne x_j$ and $y_i\ne y_j$ for $i\ne j$, because otherwise both sides of \eqref{eq5.A} vanish. We have excluded the case $n>N$ because then both sides vanish, too.

\emph{Step} 2. Let $p_0, p_1,p_2,\dots$ be the monic orthogonal polynomials with weight function $W(x)$, and let $h_k$ denote the squared norm of $p_k$:
$$
h_k:=\sum_{x\in\X}p^2_k(x)W(x), \qquad 0\le k\le N.
$$
By virtue of the classical Christoffel--Darboux identity (see e.g. \cite[\S 10.4, (10)]{Er2}),
\begin{equation}\label{eq5.H}
K_{N,W}(y,x)=\frac1{h_{N-1}}\frac{p_N(y)p_{N-1}(x)-p_{N-1}(y)p_N(x)}{y-x}\sqrt{W(y)W(x)}.
\end{equation}
Next, abbreviate 
$$
a:=a^+_{x_n}\dots a^+_{x_1}a^-_{y_1}\dots a^-_{y_n}.
$$
Then the desired equality \eqref{eq5.A} takes the form 
\begin{equation}\label{eq5.I}
(T_N(a)F,F)=\det\left[\frac1{h_{N-1}}\,\frac{p_N(x_i)p_{N-1}(y_j)-p_{N-1}(x_i)p_N(y_j)}{x_i-y_j}\, \sqrt{W(x_i)W(y_j)}\right]_{i,j=1}^n.
\end{equation}

\emph{Step} 3. We proceed to the proof of \eqref{eq5.I}. Abbreviate $X=\{x_1,\dots,x_n\}$, $Y:=\{y_1,\dots,y_n\}$. 
By virtue of Proposition \ref{prop4.B},
\begin{equation}\label{eq5.A1}
T_N(a)\de_\om=\sgn(X;Y; \om)\de_{(\om\setminus Y)\cup X},
\end{equation}
where the quantity $\sgn(X;Y;\om)$ is defined by \eqref{eq4.A}. Recall that it vanishes unless $\om$ satisfies the conditions 
\begin{equation}\label{eq5.B}
\om\supseteq Y, \qquad  (\om\setminus Y)\cap X=\varnothing.
\end{equation} 
Below we consider only such $\om$'s. 

From \eqref{eq5.A1} and the definition of $F$ it follows that 
\begin{equation}\label{eq5.B1}
(T_N(a)F,F)=\sum_{\om}(M_{N,W}(\om)M_{N,W}((\om\setminus Y)\cup X))^{1/2} \sgn(X;Y;\om)
\end{equation}
(note that $(\om\setminus Y)\cup X\in\Om_N$ due to \eqref{eq5.B}). 

Let $\Om'_{N-n}\subset\Om_{N-n}$ denote the set of those $\om'\in\Om_{N-n}$ that have empty intersection with both $X$ and $Y$. The above relation can be rewritten as
$$
(T_N(a)F,F)=\sum_{\om'\in\Om'_{N-n}}(M_{N,W}(\om'\cup X)M_{N,W}(\om'\cup Y))^{1/2} \sgn(X;Y;\om'\cup Y). 
$$

Next, we set
\begin{gather*}
V(X):=\prod_{1\le i<j\le N}(x_i-x_j), \quad  V(Y):=\prod_{1\le i<j\le N}(y_i-y_j), \\
W(\om'):=\prod_{k=1}^{N-n}W(u_k), \quad V^2(\om'):=\prod_{1\le k<l\le N-n}(u_k-u_l)^2,
\end{gather*}
where $u_1,\dots,u_{N-n}$ are the points of $\om'$ enumerated in an arbitrary order.

Using this notation and the definition of $M_{N,W}$ (see \eqref{eq5.G}) we may further write \eqref{eq5.B1} as 
\begin{multline}\label{eq5.C}
(T_N(a)F,F)=\frac{\mathcal Z_{N-n}}{\mathcal Z_N}|V(X)V(Y)| (W(X)W(Y))^{1/2}\\
\cdot \frac1{\mathcal Z_{N-n}}\sum_{\om'\in\Om'_{N-n}}W(\om')V^2(\om')\left|\prod_{u\in \om'}\prod_{i=1}^n (x_i-u)(y_i-u)\right|\sgn(X;Y;\om'\cup Y).
\end{multline}

Now we apply formula \eqref{eq4.A}, where we set $\om:=\om'\cup Y$. It shows that $\sgn(X,Y;\om'\cup Y)$ is just the sign of 
$$
V(X)V(Y)\prod_{u\in \om'}\prod_{i=1}^n (x_i-u)(y_i-u)
$$
(here we use the fact that $\sgn(a,b)$ is the sign of $a-b$). This allows us to get rid of the absolute values in  \eqref{eq5.C} and write this  relation as
\begin{multline}\label{eq5.D}
(T_N(a)F,F)=\frac{\mathcal Z_{N-n}}{\mathcal Z_N}V(X)V(Y)(W(X)W(Y))^{1/2}\\
\cdot \frac1{\mathcal Z_{N-n}}\sum_{\om'\in\Om_{N-n}}W(\om')V^2(\om')\prod_{u\in \om'}\prod_{i=1}^n (x_i-u)(y_i-u).
\end{multline}
Note that we have replaced $\Om'_{N-n}$ by the larger set $\Om_{N-n}$. This is justified because if $\om'$ is not contained in $\Om'_{N-n}$, then the double product on the right-hand side automatically vanishes. 

\emph{Step} 4. The expression on the second line of \eqref{eq5.D} is the average of the function
$$
\om'\mapsto \prod_{u\in \om'}\prod_{i=1}^n (x_i-u)(y_i-u)
$$ 
over the $(N-n)$-point orthogonal polynomial ensemble with the weight function $W$. For such a quantity there exists a general formula, see Strahov--Fyodorov \cite[Proposition 4.1]{StrFyo}. In our notation it reads as follows:
\begin{multline}\label{eq5.E}
\frac1{\mathcal Z_{N-n}}\sum_{\om'\in\Om_{N-n}}W(\om')V^2(\om')\prod_{u\in \om'}\prod_{i=1}^n (x_i-u)(y_i-u)\\
=\frac{h_{N-n}\dots h_{N-1}}{(h_{N-1})^n}\frac1{V(X)V(Y)}\det\left[\frac{p_N(x_i)p_{N-1}(y_j)-p_{N-1}(x_i)p_N(y_j)}{x_i-y_j}\right]_{i,j=1}^n
\end{multline}
(note that the quantities $N$, $K$, and $c_l$ in the formulation of \cite[Proposition 4.1]{StrFyo} correspond, in our notation, to $N-n$, $n$, and $\sqrt{h_l}$, respectively). 

Recall the well-known formula $\mathcal Z_N=h_0\dots h_{N-1}$; it yields 
\begin{equation}\label{eq5.F}
\frac{\mathcal Z_{N-n}}{\mathcal Z_N}=\frac1{h_{N-n}\dots h_{N-1}}.
\end{equation}

Finally, from \eqref{eq5.D}, \eqref{eq5.E}, and \eqref{eq5.F} we obtain after cancellations the desired equality \eqref{eq5.I}.   
\end{proof}

\section{Conditional measures}\label{sect6}

This section is a variation on the theme of Lyons \cite[\S6]{Lyons}.
Throughout the section $\X$ is a countable set with no additional structure. 

\subsection{Reduction of linear subspaces}
Introduce some notation: $E$ denotes the complex Hilbert space $\ell^2(\X)$ with its canonical orthonormal basis $\{e_x: x\in\X\}$; for an arbitrary subset $A\subseteq\X$, we denote by $E_A$ the closed subspace of $E$ spanned by the basis vectors $e_x$ with $x\in A$; the \emph{Grassmannian} $\Gr(E)$ is the set of all closed linear subspaces of $E$; likewise, we define the set $\Gr(E_A)$ for any subset $A\subset\X$; the symbol $\sqcup$ denotes disjoint union. 

\begin{definition}\label{def6.A}
Let $L\in\Gr(E)$ and let $X$ and $X'$ be two nonintersecting finite subsets of $\X$. The subspace of $E_{\X\setminus(X\sqcup X')}$ defined by
\begin{equation}\label{eq6.A}
L_{X,X'}:=(L+E_{X'})\cap E_{\X\setminus(X\sqcup X')}
\end{equation}
will be called the \emph{$(X,X')$-reduction of $L$}. An equivalent description is the following: consider the orthogonal decomposition
$$
E=E_{\X\setminus(X\sqcup X')} \oplus E_X \oplus E_{X'};
$$
then $L_{X,X'}$ consists of those vectors $v$ in the first summand, for which there exists a vector $u\in E_{X'}$ such that $v\oplus 0\oplus u\in L$. 
\end{definition}

$L_{X,X'}$ is a closed subspace, because $\dim(E_{X'})<\infty$. Thus, $L\mapsto L_{X,X'}$ is a map $\Gr(E)\to \Gr(E_{\X\setminus(X\sqcup X')})$. Note two particular cases:
\begin{equation}\label{eq6.B1}
L_{\varnothing, X'}=(L+E_{X'})\cap E_{\X\setminus X'}
\end{equation}
and
\begin{equation}\label{eq6.B2}
 L_{X,\varnothing}=L\cap E_{\X\setminus X}.
\end{equation} 
The following interpretation of \eqref{eq6.B1} is similar to \eqref{eq6.B2}:  
\begin{equation}\label{eq6.B3}
\text{$L_{\varnothing,X'}$ is the result of orthogonal projection of $L$ onto $E_{\X\setminus X'}$.}
\end{equation}

\begin{lemma}\label{lemma6.A}
{\rm(i)}  We have 
\begin{equation}\label{eq6.C1}
L_{X,X'}=(L_{\varnothing, X'})_{X,\varnothing}.
\end{equation}

{\rm(ii)} If $X'=X'_1\sqcup X'_2$, then
\begin{equation}\label{eq6.C2}
L_{\varnothing,X'}=(L_{\varnothing,X'_1})_{\varnothing,X'_2}.
\end{equation}

{\rm(iii)} Likewise, if $X=X_1\sqcup X_2$, then 
\begin{equation}\label{eq6.C3}
L_{X,\varnothing}=(L_{X_1,\varnothing})_{X_2,\varnothing}.
\end{equation}
\end{lemma}

\begin{proof}
Claim (i) directly follows from \eqref{eq6.A}. Claim (ii) directly follows from  \eqref{eq6.B3}. 
 Claim (iii) directly follows from \eqref{eq6.B2}. 
 \end{proof}

Using these relations we may decompose the reduction operation $(\ccdot)_{X,X'}$ into a composition of \emph{elementary reductions} --- those  in which one of the two sets is empty and the other is a singleton. 

\begin{remark}\label{rem6.A}
1. Our definition \eqref{eq6.A} agrees with Lyons' definition \cite[(6.3), (6.4)]{Lyons} in the case $X=\varnothing$ but not in the case $X'=\varnothing$. The difference is caused by a difference in the definition of conditional measures (see below).  

2. One can show that for any splittings $X=X_1\sqcup X_2$ and $X'=X'_1\sqcup X'_2$ one has
$$
L_{X,X'}=(L_{X_1,X'_1})_{X_2,X'_2},
$$
but we do not need this more general relation. Its proof is similar to that of \cite[Corollary 6.4]{Lyons}.
\end{remark}

\begin{definition}\label{def6.B}
Let us say that $L\in\Gr(E)$ is \emph{$(X,X')$-regular} if $L\cap E_{X'}=\{0\}$ and $L^\perp\cap E_X=\{0\}$, where $L^\perp$ denotes 
the orthogonal complement to $L$ in $E$.
\end{definition}

Obviously, if $X_1\subseteq X$, $X'_1\subseteq X'$, and $L$ is $(X,X')$-regular, then $L$ is $(X_1,X'_1)$-regular, too. 

\begin{lemma}\label{lemma6.B}
Let $L$ be $(X,X')$-regular.

{\rm(i)} For any splitting $X'=X'_1\sqcup X'_2$, the subspace $L_{\varnothing, X'_1}\subset E_{\X\setminus X'_1}$ is $(X, X'_2)$-regular.

{\rm(ii)} For any splitting $X=X_1\sqcup X_2$, the subspace $L_{X_1,\varnothing}\subset E_{\X\setminus X_1}$ is $(X_2, X')$-regular.
\end{lemma}

\begin{proof}
(i) We have to prove that $L_{\varnothing, X'_1}\cap E_{X'_2}=\{0\}$ and $(L_{\varnothing, X'_1})^\perp\cap E_X=\{0\}$. The first equality is trivial. Further, from \eqref{eq6.B3} it follows that $(L_{\varnothing, X'_1})^\perp=L^\perp\cap E_{\X\setminus X'_2}$, and from this we see that the subsequent intersection  with $E_X$ is trivial, because $L^\perp \cap E_X=\{0\}$.

(ii) Here we have to prove that $L_{X_1,\varnothing}\cap E_{X'}=\{0\}$ and $(L_{X_1,\varnothing})^\perp\cap E_{X_2'}=\{0\}$. By \eqref{eq6.B2},  $L_{X_1,\varnothing}=L\cap E_{\X\setminus X_1}$. From this, the first equality is immediate. Further, the second equality means that there are no nonzero vectors in $E_{X'_2}$ orthogonal to $L\cap E_{\X\setminus X_1}$, but this follows from the stronger condition that there are no nonzero vectors in $E_{X'}$ orthonal to $L$. 
\end{proof}

Introduce one more notation: the operator of orthogonal projection onto a subspace $L$ will be denoted by $[L]$.

\begin{lemma}\label{lemma6.C}
Let $x$ and $y$ be arbitrary points of $\X$.

{\rm(i)} Suppose that $L$ is $(\varnothing,\{y\})$-regular and write the projection $[L]$ in the block form $\begin{bmatrix}a & b\\ c & d\end{bmatrix}$ according to the orthogonal decomposition $E=E_{\X\setminus\{y\}}\oplus E_{\{y\}}$. In this notation, 
\begin{equation}\label{eq6.D1}
[L_{\varnothing, \{y\}}]=a+b(1-d)^{-1}c.
\end{equation} 

{\rm(ii)} Suppose that $L$ is $(\{x\},\varnothing)$-regular and write the projection $[L]$ in the block form $\begin{bmatrix}a & b\\ c & d\end{bmatrix}$ according to the orthogonal decomposition $E=E_{\X\setminus\{x\}}\oplus E_{\{x\}}$. In this notation, 
\begin{equation}\label{eq6.D2}
[L_{\{x\},\varnothing}]=a-bd^{-1}c.
\end{equation}
\end{lemma}

\begin{proof}
See \cite[Propositions 7.3 and 7.4]{BufOls}. Note that the assumptions in items (i) and (ii) precisely mean that $1-d\ne0$ and $d\ne0$, respectively. Consequently, the right-hand sides of \eqref{eq6.D1} and \eqref{eq6.D2} make sense. 
\end{proof}

Using the correspondence $L\mapsto [L]$, we equip $\Gr(E)$ with the topology induced by the weak operator topology on the set of projection operators. Note that for projection operators, the weak topology is the same as the strong topology. Note also that this topology is metrizable, so that it can be replaced by sequential convergence. 

\begin{lemma}\label{lemma6.D}
A sequence $\{L_N\}$ converges in the space\/ $\Gr(E)$ to some element $L$ if and only if any vectors $\xi\in L$ and $\eta\in L^\perp$ can be approximated by some sequences $\xi_N\in L_N$ and $\eta_N\in L^\perp_N$, respectively.
\end{lemma}

\begin{proof}
Easy exercise.
\end{proof}

As a corollary we see that for any fixed $(X,X')$, the set of $(X,X')$-regular subspaces is open. 

\begin{proposition}\label{prop6.A}
The $(X,X')$-reduction map $\Gr(E)\to \Gr(E_{\X\setminus(X\sqcup X')})$ is continuous on the subset of $(X,X')$-regular subspaces. 
\end{proposition}

\begin{proof}
Examine first the case of elementary reduction, when one of the sets $X$, $X'$ is empty and the other is a singleton. Then we may apply explicit formulas \eqref{eq6.D1} and \eqref{eq6.D2}. In these formulas, the number $d$ depends continuously on $[L]$. Next, the column vector $b$ depends continuously on $[L]$, too. More precisely, the map $[L]\mapsto b$ is continuous with respect to the strong topology on the set of projection operators and the norm topology on the column vectors. The same holds for $[L]\mapsto c$, because  $c=b^*$. It follows that the operators on the right-hand side of the formulas depend continuously on $[L]$. 

In the general case we decompose the map $L\mapsto L_{X,X'}$ into a composition of elementary reductions and apply the previous argument. The existence of such a decomposition follows from Lemma \ref{lemma6.A}, and Lemma \ref{lemma6.B} guarantees that the regularity assumption is preserved after each step. 
\end{proof}

\begin{remark}
If $L$ is not $(X,X')$-regular, then the continuity at $L$ may fail. This can be easily seen when $\X$ consists of two points and $\dim L=1$ (note that all arguments above hold for finite sets $\X$ as well). But the finiteness of $\X$ is not essential here; similar counterexamples can also be constructed for infinite $\X$. 
\end{remark}

\subsection{Conditional measures}
For an arbitrary subset $A\subseteq \X$, we denote by $\Om(A)$ the space $\{0,1\}^A$ whose elements are subsets of $A$. We have a natural identification  $\Om=\Om(A)\times\Om(\X\setminus A)$ and a natural projection $\Om\to \Om(\X\setminus A)$. 

By an \emph{elementary cylinder subset} of $\Om$ we mean any subset of the form
$$
C(X,X'):=\{\om\in\Om: \om\supseteq X, \om\cap X'=\varnothing\}, 
$$
where, as above,  $X\subset\X$ and $X'\subset\X$ are finite nonintersecting subsets of $\Om$. The sets $C(X,X')$ form a base of the topology on $\Om$.

The projection $\Om\to \Om(\X\setminus (X\sqcup X'))$, restricted to $C(X,X')$, induces a \emph{bijection} 
\begin{equation}\label{eq6.E}
C(X,X')\to \Om(\X\setminus (X\sqcup X')),
\end{equation}
which amounts to $\om\mapsto \om\setminus X$.

\begin{definition}\label{def6.C}
Let  $M\in\PP(\Om)$ and let $(X,X')$ be such that $M(C(X,X'))\ne0$. Denote by $M|_{C(X,X')}$ the restriction of $M$ to $C(X,X')$.
The \emph{conditional measure} $M_{X,X'}$ is defined as the pushforward of the measure $\frac1{M(C(X,X'))}M|_{C(X,X')}$ under  the bijection \eqref{eq6.E}. 
\end{definition} 

Since $\frac1{M(C(X,X'))}M|_{C(X,X')}$ is a probability measure, so is $M_{X,X'}$. 

\begin{remark}
Let us emphasize that  $M_{X,X'}$ lives on $\Om(\X\setminus (X\sqcup X'))$, not $\Om$. Note that Lyons \cite[\S6]{Lyons} treats conditional measures as measures on $\Om$,  of the form $\frac1{M(C(X,X'))}M|_{C(X,X')}$. For our purposes this definition is not convenient, because it does no permit to compare two measures, as in the proposition below.
\end{remark}

\begin{proposition}\label{prop6.B}
Let $M\in\PP(\Om)$ be such that  for any elementary cylinder $C(X,X')$, one has $M(C(X,X'))\ne0$, so that the conditional measure $M_{X,X'}$ is always defined. 

Then $M$ is $\SS$-quasi-invariant if and only if two conditional measures $M_{X,X'}$ and $M_{Y,Y'}$ are equivalent whenever $X\sqcup X'= Y\sqcup Y'$ and $|X|=|Y|$. 
\end{proposition}

\begin{proof}
Suppose that $M$ is $\SS$-quasi-invariant, and let $C(X,X')$ and $C(Y,Y')$ be two elementary cylinders such that $X\sqcup X'= Y\sqcup Y'=:A$ and $|X|=|Y|$. Then there exists a permutation $g:\X\to\X$ such that $g(X)=Y$, $g(X')=Y'$ and the action of $g$ outside $A$ is trivial. Then the induced action on $\Om$ defines a bijection $C(X,X')\to C(Y,Y')$, which reduces to the identity map on $\Om(\X\setminus A)$ after the identification
$$
C(X,X')\to\Om(\X\setminus A)\leftarrow C(Y,Y')
$$
coming from \eqref{eq6.E}. It follows that $M_{X,X'}$ and $M_{Y,Y'}$ are equivalent. 

Conversely,  suppose that any such two conditional measures are equivalent. Given an arbitrary element $g\in\SS$, there exists a finite subset $A\subset\X$ such that $g$ acts trivially on $ \X\setminus A$. Because $g$ preserves the partition
$$
\Om=\bigsqcup_{(X,X'): \;X\sqcup X'=A} C(X,X'),
$$
the above argument shows that $M$ is quasi-invariant with respect to the action of $g$. 

This completes the proof.
\end{proof}

Now we will connect conditional measures with the results of the previous subsection. Let again $L\in\Gr(E)$ and $[L]$ stand for the projection operator with the range $L$. Recall the notation $M^{[L]}$ for the corresponding determinantal measure. The next claim is analogous to that of Lyons \cite[(6.5)]{Lyons}.

\begin{proposition}\label{prop6.C}
Let $(X,X')$ be such that $L$ is $(X,X')$-regular. Then the conditional measure $M^{[L]}_{X,X'}:=(M^{[L]})_{X,X'}$ is well defined and it coincides with  $M^{[L_{X,X'}]}$.
\end{proposition}

\begin{proof}
Consider arbitrary splittings $X=X_1\sqcup X_2$ and $X'=X'_1\sqcup X'_2$. For an arbitrary measure $M\in\PP(\Om)$ we have
$$
M_{X,X'}=(M_{X_1,X'_1})_{X_2,X'_2}
$$
with the understanding that the left-hand side is defined precisely when so is the right-hand side. Combining this observation 
with Lemmas \ref{lemma6.A} and \ref{lemma6.B} we reduce the problem to the simplest case when $(X,X')$ has the form  $(\varnothing,\{y\})$ or $(\{x\},\varnothing)$ for some points $x,y\in\X$. We are going to show that then the desired result follows from  Lemma \ref{lemma6.C}. 

Indeed, examine the case of $(X,X')=(\{x\},\varnothing)$. Let us abbreviate $M:=M^{[L]}$ and let $K(\ccdot,\ccdot)$ denote the matrix of the projection $K:=[L]$. The regularity condition means that $K(x,x)>0$, which in turn means that $M(C(\{x\},\varnothing))>0$, so that the conditional measure $M_{\{x\},\varnothing}$ is defined. By the very definition, its correlation functions have the form
\begin{equation}\label{eq6.F}
\rho_n(x_1,\dots,x_n)=\frac{\det[K(x_i,x_j)]_{i,j=0}^n}{K(x_0,x_0)}, \qquad x_0:=x; \quad x_1,\dots,x_n\in\X\setminus\{x\}.
\end{equation}
Next, the ratio on right-hand side can be written as the $n\times n$ determinant $\det[\wt K(x_i,x_j)]_{i,j=1}^n$, where
$$
\wt  K(y,z):=K(y,z)-K(y,x_0)(K(x_0,x_0))^{-1}K(x_0,z), \qquad y,z\in\X\setminus \{x_0\}.
$$
Comparing this with \eqref{eq6.D2} we see that $\wt K(\ccdot,\ccdot)$ serves as a correlation kernel for $M_{\{x\},\varnothing}$, which is just what is needed (cf. Shirai--Takahashi \cite[theorem 6.5]{ST-1}). 

In the case of $(X,X')=(\varnothing,\{y\})$ the argument is similar, with a slight modification. Now we have to take, in formula \eqref{eq6.F}, the kernel 
$$
K'(y,z):=\begin{cases} 
K(y,z), & y\in\X\setminus\{x_0\}, \; z\in\X,\\
\de_{z, x_0}- K(x_0,z), & y=x_0, \; z\in\X,
\end{cases}
$$
and then apply \eqref{eq6.D1}.
\end{proof}

\section{Multiplicative functionals and conditioning}\label{sect7}

Throughout this section $\X$ is a countable set with no additional structure.

\subsection{Multiplicative functionals and their regularization}

The material of this subsection relies on Bufetov's paper \cite{Buf-2018}. We use in fact only a small part of his results, in the form suitable for our purposes, as presented in \cite{BufOls}.  

Let $\al(x)$ be a function on $\X$ such that the difference $\al(\ccdot)-1$ is in $\ell^1(\X)$, that is, $\sum_{x\in\X}|\al(x)-1|<\infty$. Then the product
$$
\Psi_\al(\om):=\prod_{x\in\om} \al(x), \qquad \om\in\Om,
$$
converges and hence defines a function on $\Om$. We call it the \emph{multiplicative functional} corresponding to $\al(\ccdot)$. In particular, $\Psi_\al$ is defined for any function $\al(x)$ such that $\al(x)-1$ is finitely supported. 

Fix a measure $M\in\PP(\Om)$ and write $\EE_M$ for the expectation with respect to $M$. If $\al(x)$ is a strictly positive function on $\X$ such that $\al(x)-1$ is finitely supported, then $\Psi_\al$ is a strictly positive cylinder function on $\Om$ and hence $\EE_M(\Psi_\al)>0$. Then we set
$$
\PPsi_{\al,M}:=\frac{\Psi_\al}{\EE_M(\Psi_\al)}.
$$

It turns out that using a limit procedure, one can extend this definition to a broader class of functions $\al(x)$.  Given a function $\al(x)$ and a finite subset $X\subset\X$, define the corresponding \emph{truncated} function by
$$
\al_X(x):=\begin{cases} \al(x), & x\in X, \\
 1, & x\in\X\setminus X. \end{cases}
$$ 

Let $\{ X\}$ denote the directed set whose elements are arbitrary finite subsets $X\subset\X$ and the partial order is defined by inclusion: $X<Y$ if $X\subset Y$. In the next result we use a limit transition along $\{X\}$. 

\begin{proposition}\label{prop7.A}
Let $\al(x)$ be a strictly positive function on $\X$ such that $\al(\ccdot)-1\in\ell^2(\X)$ and let $M\in\PP(\Om)$.  In the Banach space $L^1(\Om,M)$, there exists a limit in the norm topology
$$
\PPsi_{\al,M}:=\lim_{\{ X\}}\PPsi _{\al_X,M}.
$$
\end{proposition}

\begin{proof}
See \cite[Proposition 5.5]{BufOls}.
\end{proof}

We call $\PPsi_{\al,M}$ the \emph{normalized multiplicative functional} corresponding to $\al(\ccdot)$. Being an element of $L^1(\Om,M)$, it may be viewed as a function on $\Om$ defined modulo an $M$-null set. 

We keep to the notation introduced in Section \ref{sect6}. 
If $\mathcal L\in\Gr(E)$ and $a(x)$ is a function on $\X$ which does not vanish and tends to $1$ at infinity, then $a\mathcal L:=\{af: f\in \mathcal L\}$ is a closed subspace and hence an element of $\Gr(E)$.

\begin{proposition}\label{prop7.B}
Let $\mathcal L\in\Gr(E)$ and $a(x)$ be a nonvanishing function on $\X$ such that $a(\ccdot)-1\in\ell^2(\X)$. Then
$$
M^{[a\mathcal L]}=\PPsi_{|a|^2, \,M^{[\mathcal L]}}\cdot M^{[\mathcal L]}.
$$
\end{proposition}

Note that the function $|a(\ccdot)|^2-1$ belongs to $\ell^2(\X)$ together with $a(\ccdot)-1$. Consequently, the functional $\PPsi_{|a|^2,\, M^{[\mathcal L]}}$ is well defined by virtue of Proposition \ref{prop7.A}.

\begin{proof}
See \cite[Theorem 6.4]{BufOls}.
\end{proof}

\begin{corollary}\label{cor7.A}
Under the hypothesis of Proposition \ref{prop7.B}, the measures $M^{[\mathcal L]}$ and $M^{[a\mathcal L]}$ are equivalent.
\end{corollary}

\subsection{A continuity property for normalized multiplicative functionals}
Let  $a(x)$ be a fixed strictly positive function on $\X$ such that $a(\ccdot)-1\in\ell^2(\X)$ and let $K$ range  over the set of projection operators on $\ell^2(\X)$. For any such $K$, the function $\PPsi_{a,M^K}$ is nonnegative, hence one can define the function $\PPsi^{1/2}_{a,M^K}:=(\PPsi_{a,M^K})^{1/2}$. It belongs to $L^2(\Om,M^K)$, because the function $\PPsi_{a,M^K}$ belongs to $L^1(\Om,M^K)$ by virtue to Proposition \ref{prop6.A}. Therefore, the quantity
\begin{equation}\label{eq7.D}
\EE_{M^K}(\PPsi^{1/2}_{a,M^K})=(\PPsi^{1/2}_{a,M^K}, \; \one)_{L^2(\Om,M^K)}
\end{equation}
is well defined. The next technical lemma is used in the end of the proof of Theorem \ref{thm8.A}

\begin{lemma}\label{lemma7.A}
The quantity \eqref{eq7.D} depends continuously on $K$ with respect to the strong operator topology on the set of projection operators.
\end{lemma}

\begin{proof}
\emph{Step} 1. We claim that
$$
\EE(\PPsi^{1/2}_{a,M^K})=\lim_{\{X\}}\EE(\PPsi^{1/2}_{a_X,M^K})
$$
Indeed, let us abbreviate $f:=\PPsi_{a,M^K}$ and $f_X:=\PPsi_{a_X,M^K}$, and also $L^1:=L^1(\Om,M^K)$ and $L^2:=L^2(\Om,M^K)$. We know that $f_X$ and $f$ are nonnegative, belong to $L^1(\Om,M^K)$, and $f_X\to f$ in the $L^1$-norm (Proposition \ref{prop7.A}). Next, we use the simple inequality, which holds for any nonnegative reals $\al$ and $\be$: 
$$
(\al^{1/2}-\be^{1/2})^2\le |\al-\be|.
$$ 
It implies 
$$
\Vert f_X^{1/2}-f^{1/2}\Vert^2_{L^2}\le \Vert f_X-f\Vert_{L^1}.
$$
Therefore, $f^{1/2}_X\to f^{1/2}$ in $L^2$, which implies $(f^{1/2}_X,\;1)\to (f^{1/2},\;1)$, as desired.

\emph{Step} 2. Recall the definition of the \emph{Hilbert--Carleman regularized determinant} $\detreg(\ccdot)$, see \cite{GGK} or \cite{GK}. For a trace class operator $A$ on a Hilbert space, the definition is
$$
\detreg(1+A):=\det(1+A)\exp(-\tr A).
$$
It is well known (\cite{GGK}, \cite{GK}) that this expression is continuous in the Hilbert--Schmidt metric and extends by continuity to the space of Hilbert--Schmidt operators. Further, if $A$ is selfadjoint, then
\begin{equation}\label{eq7.A1}
\detreg(1+A)=\prod(1+\la_i)e^{-\la_i},
\end{equation}
where $\{\la_i\}$ are the eigenvalues of $A$ counted with their multiplicities (the product converges because $A$  is Hilbert--Schmidt).

We are going to prove the formula
\begin{equation}\label{eq7.A}
\EE_{M^K}(\PPsi^{1/2}_{a,M^K})=\frac{\det_2(1+(a^{1/2}-1)K)}{(\det_2(1+(a-1)K))^{1/2}}\cdot \exp\left\{\tr(([a^{1/2}-1]-\tfrac12[a-1])K)\right\}.
\end{equation}
In this step, we will show that the right-hand side of \eqref{eq7.A} is well defined. 

Indeed, both regularized determinants are defined because the operators $(a^{1/2}-1)K$ and $(a-1)K$ are Hilbert--Schmidt, which in turn follows from the assumption that $a(\ccdot)-1\in\ell^2(\X)$. Next, the same assumption implies that the function 
$$
[a^{1/2}(x)-1]-\tfrac12[a(x)-1]
$$
belongs to $\ell^1(\X)$, so that the operator $(([a^{1/2}-1]-\tfrac12[a-1])K$ is trace class and hence its trace is well defined.

It remains to show that the regularized determinant in the denominator is strictly positive. To do this, denote by $A$ the operator of multiplication by the function $a(x)$ and write $A$ in the block form $\begin{bmatrix}A_{11} & A_{12}\\ A_{21} & A_{22}\end{bmatrix}$ according to the orthogonal decomposition
$$
\ell^2(\X)=(K\ell^2(\X))\oplus((1-K)\ell^2(\X)).
$$
Then we obtain
$$
1+(a-1)K=1+(A-1)K=\begin{bmatrix}A_{11} & 0\\ A_{21} & 1\end{bmatrix},
$$
which implies the equality
$$
\detreg(1+(a-1)K)=\detreg A_{11}.
$$

From our assumption on the function $a(x)$ we see that the operator $A-1$ is Hilbert--Schmidt and, moreover, $A$ is selfadjoint and there exists $\epsi>0$ such that $A\ge\epsi$. Therefore, $A_{11}$ has the same properties.  Applying the equality \eqref{eq7.A1} (with $A$ replaced by $A_{11}-1$ we finally obtain that $\detreg A_{11}$ is strictly positive, which is just we want.

\emph{Step} 3. We proceed now to the proof of \eqref{eq7.A}. 
First, we check this formula under the additional assumption that $a(\ccdot)-1$ is finitely supported. For any function $b(x)$ such that $b(\ccdot)-1$ is finitely supported, one has 
$$
\EE_{M^K}(\Psi_b)=\det(1+(b-1)K).
$$
Applying this to $b=a^{1/2}$ and using the fact that $\Psi_a^{1/2}=\Psi_{a^{1/2}}$ we obtain
$$
\EE_{M^K}(\PPsi^{1/2}_{a,M^K})=\frac{\EE_{M^K}(\Psi_{a^{1/2}})}{(\EE_{M^K}(a))^{1/2}}=\frac{\det(1+(a^{1/2}-1)K)}{(\det(1+(a-1)K))^{1/2}}.
$$
Then this is transformed into \eqref{eq7.A} using the definition of $\det_2$. 

Now we can drop the additional assumption. Denote the right-hand side of \eqref{eq7.A} by $RHS(a)$. By virtue of step 1, it suffices to check that 
$$
\lim_{\{ X\}} RHS(a_X) =RHS(a).
$$
To see this,  we use the fact that $(a_X-1)K\to (a-1)K$ and $(a^{1/2}_X-1)K\to (a^{1/2}-1)K$ in the Hilbert-Schmidt norm, and
$$
([a_X^{1/2}-1]-\tfrac12[a-1])K \to ([a^{1/2}-1]-\tfrac12[a-1])K
$$
in the trace-class norm. 

\emph{Step} 4. Now it suffices to prove that the expression of the right-hand side of \eqref{eq7.A} is continuous with respect to the strong operator topology on the set $\{K\}$ of projection operators. 

Let $b(x)$ be a function from $\ell^2(\X)$, $K$ and  $K'$ be two projection operators, and $X\subset \X$ be a finite set.  The squared Hilbert-Schmidt norm of $bK-bK'$ can be written as
\begin{multline*}
\Vert bK-bK'\Vert_{HS}^2=\sum_{x\in\X} \sum_{y\in\X}|b(x)|^2\,|K(x,y)-K'(x,y)|^2\\=\sum_{x\in\X} \sum_{y\in\X}|b(x)|^2\,|K(y,x)-K'(y,x)|^2=\sum_{x\in \X}|b(x)|^2\Vert (K-K') e_x\Vert^2,
\end{multline*}
where the second equality holds because $K(x,y)=\overline{K(y,x)}$ and $K'(x,y)=\overline{K'(y,x)}$. 

Given a finite subset $X\subset\X$, we may split the latter sum into two parts:
\begin{equation}\label{eq7.C}
\sum_{x\in \X}|b(x)|^2\Vert (K-K') e_x\Vert^2=\sum_{x\in X}|b(x)|^2\Vert (K-K') e_x\Vert^2+\sum_{x\in\X\setminus X}|b(x)|^2\Vert (K-K') e_x\Vert^2
\end{equation}
Now, for any $\epsi>0$ we may choose $X$ so that 
$$
\sum_{x\in\X\setminus X} |b(x)|^2\le\epsi.
$$
Then the second sum on the right-hand side of \eqref{eq7.C} does not exceed $4\epsi$ while the first sum goes to $0$ as $K'$ approaches $K$ in the strong operator topology. 

This argument shows that the two regularized determinants in \eqref{eq7.A} depend continuously on $K$.

It remains to handle the third factor in \eqref{eq7.A}.  That it is, to check that the quantity $\tr(bK)$,  where
$$
b(x):=[a^{1/2}(x)-1]-\tfrac12[a(x)-1],
$$
depends continuously on $K$ in the strong operator topology on the set $\{K\}$ (which coincides with the weak operator topology). But this is easy to do using the same trick: we use the fact that the function $b(\ccdot)$ belongs to $\ell^1(\X)$ and write
$$
\tr(bK)-\tr(b K')=\sum_{x\in X}b(x)(K(x,x)-K'(x,x)) +\sum_{x\in\X\setminus X} b(x)(K(x,x)-K'(x,x)).
$$
With an appropriate choice of $X$, the second sum can be made arbitrarily small, uniformly on $K,K'$, while the first sum goes to $0$ as $K'$ approaches $K$ in the weak operator topology. 

This completes the proof.
\end{proof}

\section{Limits of discrete orthogonal polynomial ensembles}\label{sect8}

\subsection{Large-$N$ limit transition}\label{sect8.1}

We fix a countable discrete set $\X\subset\R$ and a projection operator $K$ on $\ell^2(\X)$. Let $L$ denote the range of $K$ and $K(x,y):=(K e_y,e_x)$ be the matrix of $K$.  

Next, we impose on $\X$ and $K$ the following conditions:

\smallskip

1. The series $\sum_{x\in\X}(1+|x|^2)^{-1}$ in convergent.

2. $L$ is $(X,X')$-regular for any pair $(X,X')$ of finite disjoint subsets of $\X$. 

3. There exists an infinite sequence $\X_1\subseteq\X_2\subseteq \dots$ of nested subsets of $\X$ and a sequence of kernels $K_N(x,y)$, $N=1,2,\dots$ such that:

3a. the union of all $\X_N$'s is the whole set $\X$, and in the case when the sets $\X_N$ are finite we require that $|\X_N|-N\to+\infty$;

3b. for each $N$, the kernel $K_N(x,y)$ corresponds to an $N$-point orthogonal polynomial ensemble on $\X_N$, and one has $K_N(x,y)\to K(x,y)$ pointwise as $N\to\infty$.

\smallskip

Note that, by virtue of condition 1,  $(\X,<)$ has finite intervals.

\begin{theorem}\label{thm8.A}
Under these conditions, the measure $M:=M^K$ is $\SS$-quasi-invariant, so that the state $\tau[M]$ is well defined. Furthermore,  $\tau[M]$ is the quasifree state $\varphi[K]$.
\end{theorem}

\begin{proof} Let us fix some notation: $K_N$ is the projection operator with the kernel $K_N(x,y)$; $L_N$ is its range; $M_N=M^{K_N}$ is the measure with the correlation kernel $K_N(x,y)$. In the case when $\X_N$ is a proper subset of $\X$ we embed $\ell^2(\X_N)$ into $\ell^2(\X)$ in the natural way, so that $L_N$ may be viewed as a subspace of $\ell^2(\X)$ and $K_N$ may be viewed as a projection operator acting on the whole space $\ell^2(\X)$; this enables us to treat $M_N$ as a measure on $\Om$. 

The pointwise convergence of the kernels is equivalent to the strong convergence $K_N\to K$, which in turns is equivalent to the convergence $L_N\to L$, by the definition of the topology in $\Gr(E)$. It also implies that $M_N\to M$ in the weak topology of $\PP(\Om)$. 

\smallskip

\emph{Step} 1. Let us prove that $M$ is $\SS$-quasi-invariant. Consider arbitrary finite subsets $X,X',Y,Y'$ of $\X$ such that $X\sqcup X'=Y\sqcup Y'=:Z$ and $|X|=|Y|$. Condition 2 guarantees that the hypothesis of Proposition \ref{prop6.C}  is satisfied, so that the conditional measures $M_{X,X'}$ and $M_{Y,Y'}$ are well defined. We will prove that they are equivalent. By virtue of Proposition \ref{prop6.B}, this will imply that $M$ is quasi-invariant.  

From Proposition \ref{prop6.C} we  know that $M_{X,X'}=M^{[L_{X,X'}]}$ and $M_{Y,Y'}=M^{[L_{Y,Y'}]}$. 

Define the function $a(\ccdot)$ on $\X\setminus Z$ by
\begin{equation}\label{eq8.A1}
a(u):=\dfrac{\prod_{y\in Y}(u-y)}{\prod_{x\in X}(u-x)}, \qquad u\in\X\setminus Z.
\end{equation}
Since $|X|=|Y|$, we have $a(u)=1+O(|u|^{-1})$ as $u\to\pm\infty$. Therefore, by virtue of condition 1, the function $a(u)-1$ belongs to $\ell^2(\X\setminus Z)$. Then Proposition \ref{prop7.A}  implies that for any determinantal measure $\mathcal M\in\PP(\Om(\X\setminus Z))$ admitting a projection correlation kernel, the normalized multiplicative functional $\PPsi_{|a|^2,\, \mathcal M}$ is well defined. 

We are going to prove that 
\begin{equation}\label{eq8.A}
M_{Y,Y'}=\PPsi_{|a|^2,\,\mathcal M} \cdot M_{X,X'}, \qquad \mathcal M:=M_{X,X'}.
\end{equation}
This will imply the desired claim (indeed,  from \eqref{eq8.A} it follows that $M_{Y,Y'}$ is absolutely continuous with respect to $M_{X,X'}$, and since the similar claim also holds after switching $(X,X')$ with $(Y,Y')$, the two measures are equivalent).  

Observe that an analogue of \eqref{eq8.A} holds for the pre-limit measures:
\begin{equation}\label{eq8.B}
(M_N)_{Y,Y'}=\PPsi_{|a|^2,\,\mathcal M_N} \cdot (M_N)_{X,X'}, \qquad \mathcal M_N:=(M_N)_{X,X'},
\end{equation}
provided that $N$ is large enough. 

Indeed, given $(X,X')$, the conditional measure $(M_N)_{X,X'}$ exists as soon the set $C(X,X')\cap\Om_N$ is nonempty, which holds true for large $N$ due to condition 3a.  Next, from Definition \ref{def6.C} of conditional measures $(\ccdot)_{X,X'}$ and the definition of orthogonal polynomial ensembles (see \eqref{eq5.G}) it follows that 
$$
(M_N)_{Y,Y'}=\const\cdot \Psi_{|a|^2} \cdot (M_N)_{X,X'},
$$
which implies \eqref{eq8.B}, because both $(M_N)_{Y,Y'}$ and $(M_N)_{X,X'}$ are probability measures. 

A remarkable feature of formula \eqref{eq8.B} (which plays a key role in our argument) is that the definition \eqref{eq8.A1} of the  function $a(\ccdot)$ entering this formula depends only on $X$ and $Y$, but neither on $N$ nor on the $N$th weight function. The reason, of course, is a special structure of formula \eqref{eq5.G}. 

We will deduce \eqref{eq8.A} from \eqref{eq8.B}. To do this we use the fact that
$$
(L_N)_{X,X'} \to L_{X,X'}, \qquad (L_N)_{Y,Y'} \to L_{Y,Y'},
$$
which holds by virtue of Proposition \ref{prop6.A}.

From $(L_N)_{Y,Y'} \to L_{Y,Y'}$ it follows that the left-hand side of \eqref{eq8.B} converges to the left-hand side of \eqref{eq8.A} in the weak topology of $\PP(\Om(\X\setminus Z))$. Therefore, it suffices to prove that the same holds for the right-hand sides. 

Recall the equality
$$
M^{[a \mathcal L]}=\PPsi_{|a|^2,M^{[\mathcal L]}}\ccdot M^{[\mathcal L]}
$$
from Proposition \ref{prop7.B}.  
Applying it to the subspaces $\mathcal L:=L_{X,X'}$ and $\mathcal L_N:=(L_N)_{X,X'}$ (and replacing $\X$ by $\X\setminus(X\sqcup X')$) we see that the measures on the right-hand sides of \eqref{eq8.A} and \eqref{eq8.B} can be written as  $M^{[a \mathcal L]}$ and $M^{[ a \mathcal L_N]}$, respectively. To show that $M^{[a \mathcal L_N]}\to M^{[a \mathcal L]}$ it remains to check that $ a \mathcal L_N\to  a \mathcal L$. 

We claim that this follows from the convergence $\mathcal L_N\to \mathcal L$. Indeed, recall that the function $a(\ccdot)$ is bounded together with its inverse $a^{-1}(\ccdot)$, and observe that $(a\mathcal L)^\perp=a^{-1}\mathcal L^\perp$. Then it suffices to apply Lemma \ref{lemma6.D}.

\smallskip

\emph{Step} 2. We proceed to the proof of the equality $\tau[M]=\varphi[K]$. This means that 
\begin{equation}\label{eq8.C}
\tau[M](A)=\varphi[K](A), 
\end{equation}
where $A$ is an arbitrary monomial composed from elements of the form $a^+_xa^-_y$, where $x,y\in\X$. 
We know from Theorem \ref{thm5.A} that a similar equality holds for orthogonal polynomial ensembles. It follows that 
\begin{equation}\label{eq8.D}
\tau[M_N](A)=\varphi[K_N](A)
\end{equation}
for all $N$ large enough, depending on $A$. 

Now let $A$ be fixed. We are going to derive \eqref{eq8.C} from \eqref{eq8.D} by passing to the limit in both sides of  \eqref{eq8.D} as $N\to\infty$. For the right-hand side this is trivial: we have $\varphi[K_N](A)\to \varphi[K](A)$ because $K_N\to K$. A nontrivial task is to justify the limit transition for the left-hand side, that is, to prove that 
\begin{equation}\label{eq8.E}
\lim_{N\to\infty} \tau[M_N](A)=\tau[M](A).
\end{equation}

In this step we reduce \eqref{eq8.E} to the limit relation 
\begin{equation}\label{eq8.F}
\lim_{N\to\infty}(\T[M_N](fs)\one,\one)=(\T[M](fs)\one,\one),
\end{equation}
where $f$ is a monomial formed from elements of the form $\epsi_x$ (where $x\in\X$)  and $s\in\SS$. Here we are using the notation introduced in Section \ref{sect3}: $\T[M]$ and $\T[M_N]$ are the representations of the crossed product algebras $C^*(\Om)\rtimes \SS$ and $C^*(\Om(\X_N))\rtimes \SS(\X_N)$, respectively. For given $f$ and $s$, the prelimit expression on the left makes sense for all $N$ large enough. 

Thus, suppose that the validity of \eqref{eq8.F} is already established. Then \eqref{eq8.E} is obtained easily. Indeed, from the proof of Theorem \ref{thm4.A} it follows that there exist elements $s_1,\dots,s_k\in\SS$ and polynomials $f_1,\dots,f_k$ built from elements $\epsi_x$, $x\in\X$, such that
$$
T[M](A)=\T[M](f_1 s_1)+\dots+\T[M](f_ks_k).
$$
Moreover, exactly the same relation holds when $M$ is replaced with $M_N$, provided that $N$ is large enough. This shows that (by virtue of \eqref{eq8.F})
$$
\lim_{N\to\infty}(T[M_N](A)\one,\one)=(T[M]\one,\one),
$$
which is precisely the desired relation \eqref{eq8.E}.

\smallskip

\emph{Step} 3. Let us prove \eqref{eq8.F}. Fix a finite subset $Z\subset\X$ such that the permutation $s$ acts trivially outside $Z$ and $f$ is expressed exclusively through the elements $\epsi_x$ with $x\in Z$. Given a partition $Z=X\sqcup X'$, let $\one_{X,X'}$ denote the characteristic function of the cylinder set $C(X,X')$. Note that
$$
\one=\sum_{(X,X'):\, X\sqcup X'=Z}\one_{X,X'}
$$
and that the function $f$ is a linear combination of the functions $\one_{X,X'}$. Therefore, we may assume, without loss of generality, that $f$ simply equals one of them. This enables us to eliminate $f$ and reduce \eqref{eq8.F} to the limit relation
 \begin{equation}\label{eq8.G}
\lim_{N\to\infty}(\T[M_N](s)\one_{Y,Y'},\one_{X,X'})=(\T[M](s)\one_{Y,Y'},\one_{X,X'}),
\end{equation}
where $X\sqcup X'$ and $Y\sqcup Y'$ are two arbitrary partitions of $Z$. 

Note that the transformation $\om\mapsto s^{-1}(\om)$ permutes the elementary cylinder sets  corresponding to partitions of $Z$. Then from \eqref{eq2.B} and \eqref{eq2.A} it follows that 
$$
\T[M](s)\one_{Y,Y'}(\om)=\sqrt{\dfrac{{}^s\!M}{M}(\om)}\one_{s(Y),s(Y')}(\om), \qquad \om\in\Om,
$$
and likewise
$$
\T[M_N](s)\one_{Y,Y'}(\om)=\sqrt{\dfrac{{}^s\!M_N}{M_N}(\om)}\one_{s(Y),s(Y')}(\om), \qquad \om\in\Om.
$$
It follows that both the limit and pre-limit expressions in \eqref{eq8.G} vanish unless $(s(Y),s(Y'))=(X,X')$. Thus,  we may assume that $(Y,Y')=(s^{-1}(X), s^{-1}(X'))$, and then we have
$$
(\T[M](s)\one_{Y,Y'},\one_{X,X'})=\int_{\om\in C(X,X')} \sqrt{\dfrac{{}^s\!M}{M}(\om)} M(d\om)
$$
and likewise
$$
(\T[M_N](s)\one_{Y,Y'},\one_{X,X'})=\int_{\om\in C(X,X')} \sqrt{\dfrac{{}^s\!M_N}{M_N}(\om)} M_N(d\om)
$$

Now we use the bijection $C(X,X')\to \Om(\X\setminus Z)$ in order to pass to integrals over $\Om(\X\setminus Z)$. The pushforwards, under this bijection, of the measures $M|_{C(X,X')}$ and ${}^s\!M|_{C(X,X')}$ are the measures $M(C(X,X'))\cdot  M_{X,X'}$ and $M(C(Y,Y'))\cdot  M_{Y,Y'}$, respectively. This enables us to rewrite
$$
\int_{\om\in C(X,X')} \sqrt{\dfrac{{}^s\!M}{M}(\om)} M(d\om)=\sqrt{\dfrac{M(C(Y,Y'))}{M(C(X,X'))}}\int_{\om\in\Om(\X\setminus Z)} \sqrt{\dfrac{M_{Y,Y'}}{M_{X,X'}}(\om)} M_{X,X'}(d\om).
$$
Likewise,
\begin{multline*}
\int_{\om\in C(X,X')} \sqrt{\dfrac{{}^s\!M_N}{M_N}(\om)} M_N(d\om)\\
=\sqrt{\dfrac{M_N(C(Y,Y'))}{M_N(C(X,X'))}}\int_{\om\in\Om(\X\setminus Z)} \sqrt{\dfrac{(M_N)_{Y,Y'}}{(M_N)_{X,X'}}(\om)} (M_N)_{X,X'}(d\om).
\end{multline*}

Next, since $M_N\to M$ in the weak topology, we have 
$$
\sqrt{\dfrac{M_N(C(Y,Y'))}{M_N(C(X,X'))}} \to \sqrt{\dfrac{M(C(Y,Y'))}{M(C(X,X'))}}.
$$
Therefore, the limit relation in question takes the form 
$$
\lim_{N\to\infty}\int_{\om\in\Om(\X\setminus Z)} \sqrt{\dfrac{(M_N)_{Y,Y'}}{(M_N)_{X,X'}}(\om)} (M_N)_{X,X'}(d\om)= \int_{\om\in\Om(\X\setminus Z)} \sqrt{\dfrac{M_{Y,Y'}}{M_{X,X'}}(\om)} M_{X,X'}(d\om)
$$

On the other hand, we know (see  \eqref{eq8.A} and \eqref{eq8.B}) that
$$
\dfrac{M_{Y,Y'}}{M_{X,X'}}=\PPsi_{|a|^2, M_{X,X'}}, \qquad \dfrac{(M_N)_{Y,Y'}}{(M_N)_{X,X'}}=\PPsi_{|a|^2, (M_N)_{X,X'}},
$$
where the function $a(\ccdot)$ on $\X\setminus Z$ is given by \eqref{eq8.A1}. Therefore, the desired limit relation takes the form
\begin{equation}\label{eq8.H}
\lim_{N\to\infty}\EE_{\mathcal M_N}(\PPsi^{1/2}_{|a|^2,\, \mathcal M_N})=\EE_{\mathcal M}(\PPsi^{1/2}_{|a|^2, \,\mathcal M}),
\end{equation}
where we abbreviated $\mathcal M_N:=(M_N)_{X,X'}$ and $\mathcal M:=M_{X,X'}$. 

Finally, we deduce \eqref{eq8.H} from Lemma \ref{lemma7.A}.  Recall that $(L_N)_{X,X'}\to L_{X,X'}$. This just means that the projections $[(L_N)_{X,X'}]$ converge to the projection $[L_{X,X'}]$. On the other hand, 
$$
\mathcal M_N=M^{[(L_N)_{X,X'}]}, \quad \mathcal M=M^{[L_{X,X'}]}.
$$
Therefore, \eqref{eq8.H} follows from Lemma \ref{lemma7.A}. 
\end{proof}

\subsection{Examples of perfect limit measures}\label{sect8.2}

Theorems \ref{thm5.A} and \ref{thm8.A} are results of general character, and now we turn to concrete examples. Let us focus on hypergeometric and $q$-hypergeometric orthogonal polynomials listed in the Askey scheme and $q$-Askey scheme \cite{KLS-2010}. Recall that Theorem \ref{thm5.A} is applicable to systems of discrete orthogonal polynomials with a constraint on the support $\X$ of the weight function ($\X$ must have finite intervals). This constraint cuts out some systems with infinite $\X$ (for instance, big $q$-Jacobi polynomials). The systems satisfying the constraint are the following. 

\smallskip

$\bullet$ \emph{Infinite $\X$}: Meixner, Charlier, little $q$-Jacobi, little $q$-Laguerre, alternative $q$-Charlier, $q$-Charlier, Al-Salam-Carlitz II. 

\smallskip

$\bullet$ \emph{All systems with finite $\X$}:  Racah, Hahn, Krawtchouk,  $q$-Racah, $q$-Hahn, dual $q$-Hahn, quantum $q$-Krawtchouk, $q$-Krawtchouk, affine $q$-Krawtchouk, dual $q$-Krawtchouk.

\smallskip

The key property of hypergeometric and $q$-hypergeometric polynomials is that they are eigenfunctions of some second order difference (or $q$-difference) operators. This property underlies a simple operator method, which allows one to compute large-$N$ limits of Christoffel-Darboux kernels without appeal to complicated asymptotic formulas for orthogonal polynomials, see \cite{BO-2007}, \cite{Ols-2008}. The method was developed in detail in \cite{BO-2017}. As shown in that paper, one can obtain by this method three families of determinantal measures with the correlation kernels on $\Z_{\ge0}\times\Z_{\ge0}$ of the form
\begin{equation}\label{eq8.I}
K^+_r(x,y)=\int_{r}^{+\infty} \wt{\mathcal P}_x(t)\wt{\mathcal P}_y(t)\mathcal W(dt), \qquad K^-_r(x,y)=\int_{-\infty}^r \wt{\mathcal P}_x(t)\wt{\mathcal P}_y(t)\mathcal W(dt),
\end{equation}
where $x,y\in\Z_{\ge0}$; $\{\wt{\mathcal P}_x: x=0,1,2,\dots\}$ is one of the three systems of classical continuous orthogonal polynomials --- Hermite, Laguerre or Jacobi (the polynomials $\wt P_x$ are assumed to be orthonormal and with positive leading coefficients); $\mathcal W$ is the corresponding weight function; $r$ is a parameter, which is assumed to be inside the support of $\mathcal W$. These kernels are called the \emph{discrete Hermite, Laguerre, and Jacobi kernels}, respectively. By the very definition, these are projection kernels.

For the kernels $K^\pm_r(x,y)$ there exists an alternative explicit formula applicable for $x\ne y$. For instance, in the Hermite case it looks as follows:
$$
K^\pm_r(x,y)=\mp(\pi x!y! 2^{x+y+2})^{-1/2}e^{-r^2}\frac{H_{x+1}(r)H_y(r)-H_x(r)H_{y+1}(r)}{x-y}, 
$$
where $H_x(\ccdot)$ denotes the Hermite polynomial of degree $x=0,1,2,\dots$, in the usual standardization, as in \cite{Er2} and \cite{Sz}. 

Note that the discrete Laguerre kernel involves an additional continuous parameter and the discrete Jacobi kernel involves two additional parameters, see \cite[Section 3]{BO-2017}. However, for the sake of brevity, we suppress the additional parameters and use a uniform notation $K^\pm_r$ for all three families. Let $M^\pm_r$ denote the corresponding determinantal measures on\/ $\Om=\{0,1\}^{\Z_{\ge0}}$.

\begin{theorem}\label{thm8.B}
{\rm(i)}  The determinantal measures $M^\pm_r$ with the discrete Hermite, discrete Laguerre or discrete Jacobi kernels are perfect measures in the sense of Definition \ref{def4.B}.  

{\rm(ii)} The canonical correlation kernels of the measures $M^+_r$ and $M^-_r$ have the following form.

$\bullet$ In the case of Hermite and Jacobi, these are $K^+_r(x,y)$ and $(-1)^{x-y}K^-_r(x,y)$, respectively. 

$\bullet$ In the case of Laguerre, these are $(-1)^{x-y}K^+_r(x,y)$ and $K^-_r(x,y)$, respectively. 
\end{theorem}

\begin{proof}
Let us check conditions 1-3 formulated in the beginning of \S \ref{sect8.1}. Condition 1 holds because the series  $\sum_{x=0}^\infty(1+x^2)^{-1}$ converges.  

Condition 2 means that $L$ and $L^\perp$ must have trivial intersection with any finite-dimensional linear subspace of the form $E_X\subset E$, where $X$ is a finite subset of $\X=\Z_{\ge0}$. In our situation $E$ is the coordinate Hilbert space $\ell^2(\Z_{\ge0})$, that is, the space of square summable sequences $(c_0,c_1,\dots)$ of complex numbers. The subspaces $L$ and $L^\perp$ consist of sequences satisfying one of the two conditions 
$$
\mathcal W(t)^{1/2}\sum_{n=0}^\infty c_n \wt P_n(t)\equiv 0 \quad \text{for all $t\ge r$ or for all $t\le r$}.
$$
Neither of them can hold for a nontrivial \emph{finite} linear combination of polynomials $\wt P_n(t)$, because a nonzero polynomial cannot vanish identically on an interval of $\R$. This proves that condition 2 is satisfied. 

Condition 3 (approximation by orthogonal polynomial ensembles) is established in \cite{BO-2017}. In more detail, the following limit transition hold:

$\bullet$   $N$-particle Charlier ensembles $\to$ discrete Hermite kernel measures with parameters $(+,r)$,  \cite[Theorem 6.1]{BO-2017};

$\bullet$  $N$-particle Meixner ensembles $\to$ discrete Laguerre kernel measures with parameters $(-,r)$,    \cite[Theorem 6.2]{BO-2017};

$\bullet$  $N$-particle Racah ensembles $\to$  discrete Jacobi kernel measures with parameters $(+,r)$, \cite[Theorem 6.6]{BO-2017}.  

Finally, it remains to handle the measures with the complementary kernels, that is, the discrete Hermite and Jacobi kernels with parameters $(-,r)$ and the discrete Laguerre kernel with parameters $(+,r)$, and this is done with the aid of  Proposition \ref{prop4.E}. 
\end{proof}

One more example is the determinantal measure $M^\dSine_\phi\in\PP(\{0,1\}^\Z)$ defined by the  \emph{discrete sine kernel} on $\Z\times\Z$. Recall the formula for that kernel:
\begin{equation}\label{eq8.J}
K^\dSine_\phi(x,y):=\begin{cases} \dfrac{\sin( \phi\,(x-y))}{\pi(x-y)},  & x\ne y,\\
\phi/\pi,  & x=y, \end{cases}
\end{equation}
where $\phi\in(0,\pi)$ is a parameter. The corresponding operator on $\ell^2(\Z)$, denoted by $K^\dSine_\phi$, is a projection operator: under the Fourier transform $\ell^2(\Z)\to L^2(\mathbb T)$ (where $\mathbb T$ denotes the unit circle equiped with the uniform measure), $K^\dSine_\phi$ turns into the operator of multiplication by $\one_{(-\phi,\phi)}$ --- the characteristic function of the arc of length $2\phi$ centered at the point $1\in\mathbb T$.

\begin{theorem}\label{thm8.C}
The determinantal measure $M^\dSine_\phi\in\PP(\{0,1\}^\Z)$ with the correlation kernel kernel $K^\dSine_\phi(x,y)$, is perfect, and $K^\dSine_\phi(x,y)$ is its canonical correlation kernel. 
\end{theorem}

\begin{proof}
As in the previous theorem, we apply Theorem \ref{thm8.A}, and our task is to check again conditions 1-3. Condition 1 is obvious. Condition 2 follows from the fact that a trigonometric polynomial cannot vanish identically on an arc. For condition 3 we need to find an appropriate approximation by discrete orthogonal polynomial ensembles. This can be done in various ways; apparently the simplest one is to use the Charlier polynomials. These are orthogonal polynomials on $\Z_{\ge0}$ with the weight function $W(x)=\theta^x/x!$, where $\theta>0$ is a parameter. For our purpose any value of $\theta$ is suitable; let us take $\theta=1$. 

The rest of the proof consists in application of the algorithm described in \cite{BO-2017}. We easily find a difference operator $D_N$ on $\Z_{\ge0}$ such that the $N$th Charlier kernel $K_N(x,y)$ corresponds to the projection on the positive part of the spectrum of $D_N$. Namely, $D_N$ acts on a test function $f(x)$ by
$$
D_N f(x)=\sqrt{x+1} f(x+1)+\sqrt x f(x-1)+(-x+N-1)f(x),
$$
cf. \cite[\S6.1]{BO-2017}. Then we replace $x$ by $x+s_N$, where $s_N$ is a large positive integer depending on $N$ ($s_N$ will be specified shortly), which amounts to shifting the lattice $\Z_{\ge0}$ to the left by $s_N$. Thus, in the notation of Theorem \ref{thm8.A}, we have $\X_N=\{-s_N, -s_N+1,-s_N+2, \dots\}$. Finally, we divide the operator by $\sqrt{s_N}$, which does affect the spectral projection corresponding to the positive part of the spectrum. The resulting difference operator has the form 
$$
\wt D_N f(x)=\sqrt{\dfrac{x+1+s_N}{s_N}} f(x+1)+\sqrt{\dfrac{x+s_N}{s_N}}  f(x-1)+\frac{-x+N-1}{\sqrt{s_N}}f(x).
$$

Now we can specify the shifts $s_N$: we set $s_N:=[N+(2\cos\phi)\sqrt N]$, where $[\ccdot]$ denotes the integer part.  Under this choice, $\wt D_N$ converges coefficient-wise to the difference operator 
$$
D f(x)= f(x+1)+f(x-1)-2\cos\phi f(x),
$$
which is just we need. 

Indeed, the spectral projection on the positive part of the spectrum of $D$ is precisely the operator $K^\dSine_\phi$ (see, e.g., \cite{Ols-2008}).  Next, we have 
$$
(N+1+(2\cos\phi)\sqrt{N+1})-(N+(2\cos\phi)\sqrt N)=1+\frac{2\cos\phi}{\sqrt{N(N+1)}}.
$$
Since $2\cos\phi>-2$, the right-hand side is positive for any $\phi$ (excluding only $N=1$), so that $s_N\le s_{N+1}$ and hence the condition $\X_N\subseteq\X_{N+1}$ is satisfied. 

This completes the proof. 
\end{proof}

\section{Basic facts about quasifree states}\label{sect9}

Let $E$ be a separable complex Hilbert space. Recall the notation  $\Gr(E)$ for the set of closed linear subspaces of $E$. 
Next, let $\mathscr C_2(E)$ be the set of Hilbert-Schmidt operators on $E$ and $U_2(E)$ be the group of unitary operators that differ from the identity by a Hilbert-Schmidt operator. The group $U_2(E)$ acts on $\Gr(E)$ in a natural way.  

\begin{proposition}\label{prop9.A}
Let $L_1,L_2\in\Gr(E)$ and $K_1,K_2$ denote the corresponding projection operators. The following two conditions are equivalent:

{\rm(i)} $L_1$ and $L_2$ lie on the same $U_2(E)$-orbit;

{\rm(ii)} $K_1-K_2\in\mathscr C_2(E)$ and the operator $K_2K_1: L_1\to L_2$ has index\/ $0$.
\end{proposition}

Note that  the first condition in (ii) implies that $K_2K_1: L_1\to L_2$ is a Fredholm operator, so that its index (the difference between the dimensions of the kernel and cokernel) is well defined. Note also that 
\begin{equation}\label{eq9.A}
\operatorname{index}(K_2K_1: L_1\to L_2)=\dim \Ker (K_2K_1\big|_{L_1})-\dim \Ker (K_1K_2\big|_{L_2}).
\end{equation}

\begin{proof}
See Str\u{a}til\u{a}--Voiculescu \cite[\S 3.10]{SV-1978}.
\end{proof}

\begin{theorem}\label{thm9.A}
Let, as above, $\X$ be a countable set, $\A^0$ be the gauge invariant subalgebra of the\/ $\CAR$ algebra $\A=\A(\X)$, and $E:=\ell^2(\X)$.

{\rm(i)} A quasifree state\/ $\varphi[K]$ on the algebra\/ $\A^0$ (see Definition \ref{def1.A}) is a pure state if and only if $K$ is a projection operator.

{\rm(ii)} Let $K_1$ and $K_2$ be two projection operators and $L_1$ and $L_2$ be their ranges. The states $\varphi[K_1]$ and $\varphi[K_2]$ are equivalent if and only if $L_1$ and $L_2$ lie on the same $U_2(E)$-orbit (equivalently, $K_1$ and $K_2$ satisfy the second condition of Proposition \ref{prop9.A}). 
\end{theorem}

\begin{proof}
See Str\u{a}til\u{a}--Voiculescu \cite[\S 3.1, claims $2^\circ$ and $4^\circ$]{SV-1978}
\end{proof}

Note that a similar result is formulated in Baker \cite[p. 38]{Baker} without proof --- with reference to the unpublished thesis of G. Stamatopoulos (University of Pennsylvania, 1974).  In \cite{Baker},  the criterion of equivalence is given in the following form: 
\begin{equation}\label{eq9.B}
\tr((1-K_1)K_2(1-K_1))=\tr((1-K_2)K_1(1-K_2))<\infty.
\end{equation}
The fact that condition \eqref{eq9.B} is equivalent to the second condition of Proposition \ref{prop9.A} can be readily deduced from a nice classical result on the relative position of two linear subspaces in a Hilbert space, to which proceed now. 

Let $L_1, L_2\in\Gr(E)$,  and $K_1,K_2$ be the corresponding projection operators. The space $E$ can be decomposed into a direct sum of five subspaces,
\begin{equation}\label{eq9.C}
(L_1\cap L_2)\oplus(L_1\cap L_2^\perp)\oplus (L_1^\perp\cap L_2)\oplus(L_1^\perp\cap L_2^\perp)\oplus L,
\end{equation}
where, as before,  $(\ccdot)^\perp$ denotes the orthogonal complement to a given subspace and $L$ is the orthogonal complement to the sum of the first four subspaces. Each of these five subspaces is invariant under the action of the projections $K_1$ and $K_2$. Their action on the first four subspaces is evident, and the action on $L$ can be described in the following way.

\begin{proposition}\label{prop9.B}
The triple $(L, K_1,K_2)$ is unitarily equivalent to a triple of the form $(V\oplus V, K'_1,K'_2)$, where $V$ is a Hilbert space and $K'_1$ and $K'_2$ are two projections written in the block form as
\begin{equation}\label{eq9.D}
K'_1=\begin{bmatrix} 1 & 0\\0 & 0\end{bmatrix}, \qquad K'_2=\begin{bmatrix}  C^2 & CS \\ CS & S^2\end{bmatrix},
\end{equation}
where $C$ and $S$ are two commuting selfadjoint operators on $V$ such that $0\le C\le1$, $0\le S\le 1$, $C^2+S^2=1$, and $\Ker C=\Ker S=0$. 
\end{proposition}

\begin{proof}
This result is due to Dixmier \cite{Dixmier}. An elegant proof is given in Halmos \cite{Halmos}.  See also the survey  B\"{o}ttcher--Spitkovsky \cite{BS-2010}.
\end{proof}

Note that the spectral types of $C$ and $S$ are invariants of the pair $L_1,L_2$, and they depend symmetrically on $L_1$ and $L_2$. This follows, for instance, from the identity
$$
(K'_1+K'_2-1)^2=\begin{bmatrix} C^2 & 0\\ 0 & C^2 \end{bmatrix}
$$
in \cite[p. 386]{Halmos}.  

Proposition \ref{prop9.B} leads to a useful reformulation of condition (ii) in Proposition \ref{prop9.A}:

\begin{proposition}\label{prop9.C}
Let, as above, $L_1,L_2\in\Gr(E)$ and $K_1,K_2$ be the corresponding projections. 

The condition $K_1-K_2\in\mathscr C_2(E)$ is equivalent to the following: first, the subspaces $L_1\cap L_2^\perp$ and $L_1^\perp\cap L_2$ have finite dimension and, second, $S\in \mathscr C_2(V)$. 

Next, if this holds true, then the index of the operator $K_2K_1: L_1\to L_2$ equals $0$ if and only if $\dim(L_1\cap L_2^\perp)=\dim(L_1^\perp\cap L_2)$. 
\end{proposition}

\begin{proof}
From Proposition \ref{prop9.B} it is seen that $K_1-K_2\in\mathscr C_2(E)$ if and only if, first, the subspaces $L_1\cap L_2^\perp$ and $L_1^\perp\cap L_2$ have finite dimension and, second, $K'_1-K'_2\in \mathscr C_2(L)$. 

Now we claim that the latter condition is equivalent to $S\in \mathscr C_2(V)$. Indeed, if $K'_1-K'_2\in \mathscr C_2(L)$, then $S^2\in\mathscr C_2(V)$ and $CS\in\mathscr C_2(V)$. Looking at the eigenvalues of $S$ we see that this just means that $S\in\mathscr C_2(V)$. The inverse implication is also evident. 

Let $L'_1$ and $L'_2$ be the ranges of the projections $K'_1$ and $K'_2$, respectively. We claim that the operator  $K'_2 K'_1\big|_{L'_1}$ has trivial kernel. Indeed, we have 
$$
(K'_2 K'_1)^*(K'_2K_1)\big|_{L'_1}=K'_1K'_2K'_1\big|_{L'_1}=C^2.
$$
But $C^2$ has trivial kernel, because so does $C$. Likewise, the operator $K'_1 K'_2\big|_{L'_2}$ has trivial kernel, too.  Applying \eqref{eq9.A} we see that the index of $K'_2K'_1: L'_1\to L'_2$ equals $0$. 
It follows that 
$$
\operatorname{index}(K_2K_1: L_1\to L_2)=\dim(L_1\cap L_2^\perp)-\dim(L_1^\perp\cap L_2).
$$
This proves the last statement of the proposition.
\end{proof}

As a corollary we obtain a sufficient condition for the vanishing of the index. 

\begin{corollary}\label{cor9.A}
Let $K_1$ and $K_2$ be projection operators on a Hilbert space $E$ and $L_1,L_2$ be their ranges. Suppose that  $K_1-K_2\in\mathscr C_2(E)$. 

If\/ $\Vert K_1-K_2\Vert<1$, then the operator 
$K_2K_1: L_1\to L_2$ has index $0$.
\end{corollary}

\begin{proof}
Since $\Vert K_1-K_2\Vert<1$, the subspaces $L_1\cap L_2^\perp$ and $L_1^\perp\cap L_2$ are null. Then the last statement of Proposition \ref{prop9.C} shows that the index is $0$. 
\end{proof}

\begin{remark}\label{rem9.A}
The following argument, based on Corollary \ref{cor9.A}, shows that in certain circumstances one may avoid the explicit computation of the operator $K_2K_1: L_1 \to L_2$. 

Namely, suppose $t$ range over an interval $I\subset\R$ and $\{K(t): t\in I\}$ is a family of projection operators on a Hilbert space $E$ depending continuously on $t$ with respect to the norm topology. Suppose further that $K(t_1)-K(t_2)\in\mathscr C_2(E)$ for any $t_1, t_2\in I$.  Then, by virtue of Corollary \ref{cor9.A}, the index related to  any pair $K_1=K(t_1)$, $K_2=K(t_2)$ is $0$. 

We use this argument below in the proof of Theorem \ref{thm10.B}. 
\end{remark}

\section{Applications and remarks}\label{sect10}

\begin{proposition}\label{prop10.A}
Let $M$ be a perfect measure such that the corresponding canonical kernel $K(x,y)$ is a projection kernel (see Definition \ref{def4.B}). Then $M$ is ergodic.
\end{proposition}

\begin{proof}
By Proposition \ref{prop2.B}, an $\SS$-quasi-invariant measure $M$ is ergodic if and only if the associated representation $\T[M]$ of the algebra $C(\Om)\rtimes\SS$ is irreducible. Thus, we have to show that $\T[M]$ is irreducible. Next, we may  replace $\T[M]$ by the representation $T[M]$ of the algebra $\A^0$. Since the distinguished vector $\one\in L^2(\Om,M)$ is a cyclic vector for $T[M]$, this representation is irreducible if and only if $\tau[M]$ is a pure state. Since $M$ is assumed to be perfect, this means that $\varphi[K]$ is a pure state (here $K$ is the operator corresponding to $K(x,y)$. Finally, by Theorem \ref{thm9.A} (ii) this holds true because $K$ is a projection.
\end{proof}

Here is a direct corollary:

\begin{corollary}\label{cor10.A}
The perfect measures from Theorems \ref{thm8.B} and \ref{thm8.C} are ergodic measures.
\end{corollary}

\begin{remark}\label{rem10.A}
A completely different possible way to establish the ergodic property is indicated in Bufetov \cite[remark after Theorem 1.6]{Buf-2018}: an $\SS$-quasi-invariant determinantal measure on $\{0,1\}^\X$ is ergodic provided it it is \emph{number rigid} in the sense of Ghosh--Peres \cite{Ghosh}, \cite{GhoshPeres}. The idea is to combine the rigidity property with Lyons' theorem on triviality of the tail $\si$-algebra, which holds for any determinantal measure \cite[Theorem 7.15]{Lyons}.
\end{remark}

\begin{proposition}\label{prop10.B}
Let $M_1$ and $M_2$ be two perfect measures and $K_1(x,y)$ and $K_2(x,y)$ be their  canonical correlation kernels. Suppose that they are projection kernels, and let $K_1$ and $K_2$ denote the corresponding projection operators.

The following dichotomy holds: $M_1$ and $M_2$ are either equivalent or disjoint, and this happens depending on whether the corresponding quasifree states $\varphi[K_1]$ and $\varphi[K_2]$ are equivalent or not. 
\end{proposition}

\begin{proof}
The argument is the same as in Proposition \ref{prop10.A}, with reference to Proposition \ref{prop2.A} and Proposition \ref{prop2.B} (i).
\end{proof}

\begin{remark}
Note that for rigid perfect measures, the projection property for the canonical correlation kernel holds automatically. Indeed, let $M\in\{0,1\}^\X$ be a perfect measure, so that $\tau[M]=\varphi[K]$. If one knows additionally that $M$ is rigid, then, according to Remark \ref{rem10.A}, $M$ is $\SS$-ergodic. But then  $\varphi[K]$ is a pure state, which in turn implies that $K$ is a projection, by virtue of Theorem \ref{thm9.A} (i). 
\end{remark}

\begin{corollary}[Criterion of equivalence/disjointness]\label{cor10.B}
Let $M_1$ and $M_2$ be two measures satisfying the hypotheses of Proposition \ref{prop10.B}. If the corresponding projections $K_1$ and $K_2$ satisfy the condition of Theorem \ref{thm9.A} (ii), then  $M_1$ and $M_2$  are equivalent; otherwise they are disjoint.
\end{corollary}

In the rest of the paper we give two examples of application of this criterion (Theorems \ref{thm10.A} and \ref{thm10.B}). 

Let $M^\dHermite_{+,r}$, where $r\in\R$, denote the determinantal measure on $\{0,1\}^{\Z_{\ge0}}$ with the discrete Hermite kernel with parameters $(+,r)$, see \S \ref{sect8.2} above. 

\begin{theorem}\label{thm10.A}
The measures $M^\dHermite_{+,r}$, where $r$ ranges over\/ $\R$, are pairwise disjoint.
\end{theorem}

\begin{proof}
By Theorem \ref{thm8.B}, these are perfect measures. Next, the same theorem tells us that their canonical correlation kernels are the discrete Hermite kernels. Let $K^\dHermite_{+,r}$ denote the corresponding projection operators. From \eqref{eq8.I} it follows that if $r_1<r_2$, then the difference $K^\dHermite_{+,r_1}-K^\dHermite_{+,r_2}$ is a projection operator with infinite-dimensional range. Such an operator is not Hilbert-Schmidt. Therefore, by the criterion of Theorem \ref{thm9.A}, the corresponding quasifree states are not equivalent. This completes the proof. 
\end{proof}

\begin{remark}
Sometimes one can prove that two determinantal measures are disjoint when their first correlation functions have different asymptotics at infinity. (The idea is to use the strong law of large numbers for the size of the particle configuration in a growing interval, of the type of \cite[Theorem 5.1]{Ols-98}.)  For instance, this test easily shows that the measures $M^\dSine_\phi$ with different values of parameter $\phi$ are disjoint. 

However,  the result of Theorem \ref{thm10.A}  cannot be obtained by such a rough method.  Indeed, in the situation of Theorem \ref{thm10.A}, the first correlation function of $M^\dHermite_{+,r}$ tends at infinity to the constant $\frac12$, irrespective of the value of the parameter $r$ (this can be checked using the approach of \cite{BO-2017}), so that the above test is not sensitive to a change in parameter $r$. 
\end{remark}

Consider the classical Jacobi polynomials $\J_n(t)$ with equal parameters $(a,a)$, where $a>-1$. These are orthogonal polynomials  with the weight measure $(1-t^2)^a dt$ on $[-1,1]$. We use the same standardization as in \cite{Er2} and \cite{Sz}: it is specified by the condition that the leading coefficient of $\J_n(t)$ equals
$$
k_n:= 2^{-n}\binom{2n+2a}{n}.
$$
Note that in a different standardization, these polynomials turn into the ultraspherical polynomials $C^\la_n(t)$ with parameter $\la=a+\frac12$. 

Given $a>-1$, we form two kernels on $\Z_{\ge0}\times\Z_{\ge0}$,
\begin{gather*}
K^{+,a}(x,y):=\frac1{\Vert \J_x\Vert\, \Vert\J_y\Vert}\,\int_0^{+\infty} \J_x(t)\J_y(t)(1-t^2)^adt, 
\\  K^{-,a}(x,y):= \frac1{\Vert \J_x\Vert\, \Vert\J_y\Vert}\,  \int_{-\infty}^0 \J_x(t)\J_y(t)(1-t^2)^adt.
\end{gather*}
These are projection kernels of the form \eqref{eq8.I} with $r=0$ (our definition agrees with \eqref{eq8.I} because $k_n>0$ for all $n\in\Z_{\ge0}$), and they are a particular case of the discrete Jacobi kernels \cite{BO-2017}. Let us denote by $M^{+,a}$ and $M^{-,a}$ the corresponding determinantal measures; they are an instance of measures denoted in Section \ref{sect8.2} by $M^+_r$ and $M^-_r$. 

Theorem \ref{thm8.B} shows that  the measures $M^{+,a}$ and $M^{-,a}$ are perfect, and their canonical correlation kernels are $K^{+,a}(x,y)$ and $(-1)^{x-y}K^{-,a}(x,y)$, respectively. Moreover, $M^{-,a}=(M^{+,a})^\circ$ (recall that $(\ccdot)^\circ$ is the particle/hole involution, see \S \ref{sect4.5}).   
On the other hand, because of the symmetry relation $\J_n(-t)=(-1)^n\J_n(t)$, the kernels $K^{+,a}(x,y)$ and $(-1)^{x-y}K^{-,a}(x,y)$ coincide. It follows that $M^{+,a}=(M^{+,a})^\circ$; in words: the measure $M^{+,a}$ coincides with its pushforward under the particle/hole involution. One more consequence is that 
\begin{equation}\label{eq10.A}
K^{+,a}(x,x)\equiv\tfrac12.
\end{equation}
In what follows we denote by $K^{+,a}$ the projection operator on $E:=\ell^2(\Z_{\ge0})$ with the kernel $K^{+,a}(x,y)$.

\begin{theorem}\label{thm10.B}
All the measures $M^{+,a}$, $a>-1$, are equivalent. 
\end{theorem}

\begin{proof}
Because the measures are perfect and we know their canonical correlation kernels, we may apply the criterion of Corollary \ref{cor10.B}. It tells us that we have to check condition (ii) of Theorem \ref{thm9.A}. By virtue of Remark \ref{rem9.A} it suffices to prove that for any $a_1,a_2>-1$ the operator $K^{+,a_1}-K^{+,a_2}$ is Hilbert--Schmidt and its Hilbert--Schmidt norm depends continuously on the parameters (then it will follow that the last claim holds for the ordinary operator norm as well, because it is majorated by the Hilbert--Schmidt norm). 

Thus, we have to prove that
$$
\sum_{x,y\in\Z_{\ge0}}(K^{+,a_1}(x,y)-K^{+,a_2}(x,y))^2<\infty
$$
and, moreover, the sum depends continuously on $a_1, a_2$. 

Because of \eqref{eq10.A}, the summands with $x=y$ vanish, so that we may assume $x\ne y$. Then we use the formula for the discrete Jacobi kernel given in \cite[Proposition 3.5]{BO-2017}, where we set $b=a$ and $r=0$. It gives us
\begin{multline}\label{eq10.B}
K^{+,a}(x,y)=\frac1{2\Vert \J_x\Vert\, \Vert\J_y\Vert}\\
\times \frac{(x+2a+1)P^{(a+1,a+1)}_{x-1}(0)\J_y(0)-\J_x(0)(y+2a+1)P^{(a+1,a+1)}_{y-1}(0)}{(x-y)(x+y+2a+1)}.
\end{multline}

Let us write down explicit expressions for the quantities entering \eqref{eq10.B}: 
\begin{equation}\label{eq10.C}
\Vert \J_k\Vert^{-1}=\frac{\sqrt{(k+a+\frac12)\Ga(k+2a+1)\Ga(k+1)}}{2^a\Ga(k+a+1)}
\end{equation}
(see \cite[(4.3.3)]{Sz}) and
\begin{equation}\label{eq10.D}
\J_k(0)=\begin{cases} 0, & \text{$k$ odd;}\\ \dfrac{(-1)^\ell\Ga(k+a+1)}{2^k \Ga(\ell+1)\Ga(\ell+a+1)}, & \text{$k$ even, $k=2\ell$.}  \end{cases}
\end{equation}
The latter formula can be obtained, e.g. as follows: we use the explicit expression 
$$
\J_k(t)=2^{-k}\sum_{i=0}^k\binom{k+a}{k-i}\binom{k+a}{i}(t-1)^i(t+1)^{k-i}
$$
(see \cite[(4.3.2)]{Sz});  by specializing $t=0$ we see that $2^k\J_k(0)$ equals the coefficient of $u^k$ in the series expansion of
$$
(1-u)^{k+a}(1+u)^{k+a}=(1-u^2)^{k+a},
$$
which gives \eqref{eq10.D}.

From \eqref{eq10.D} it is seen that \eqref{eq10.B} vanishes unless one of the variables $x,y$ is odd and the other is even. Because of the symmetry of the kernel it suffices to examine the case when $x$ is odd and $y$ is even. Now we change the notation and set
$$
x=2n+1, \quad y=2m,  \qquad n,m\in\Z_{\ge0}.
$$
Taking into account \eqref{eq10.D} we obtain from \eqref{eq10.B}
\begin{equation}\label{eq10.E}
K^{+,a}(2n+1,2m)=\frac{n+a+1}{4\Vert \J_{2n+1}\Vert\, \Vert\J_{2m}\Vert}
\frac{P^{(a+1,a+1)}_{2n}(0)\J_{2m}(0)}{(n-m+\frac12)(n+m+a+1)}.
\end{equation}

\begin{lemma}\label{lemma10.A}
We have
\begin{equation}\label{eq10.F}
\frac{(n+a+1)P^{(a+1,a+1)}_{2n}(0)\J_{2m}(0)}{\Vert \J_{2n+1}\Vert\, \Vert\J_{2m}\Vert}=4(n+1)\left(1+O\left(\tfrac1{n+1}\right)\right)\left(1+O\left(\tfrac1{m+1}\right)\right)
\end{equation}
where the estimate on the right-hand side is uniform with respect to parameter $a>-1$ provided it is bounded away from $-1$ and $+\infty$. 
\end{lemma}

\begin{proof}[Proof of the lemma]
Using \eqref{eq10.C} and \eqref{eq10.D} we obtain (after a minor cancellation) that the left-hand side of \eqref{eq10.F} equals
\begin{gather*}
\frac{(-1)^n\sqrt{(2n+a+\frac32)\Ga(2n+2a+2)\Ga(2n+2)}}{2^{2n+a}\Ga(n+1)\Ga(n+a+1)}\\
\times
\frac{(-1)^m\sqrt{(2m+a+\frac12)\Ga(2m+2a+1)\Ga(2m+1)}}{2^{2m+a}\Ga(m+1)\Ga(m+a+1)}.
\end{gather*}
Applying the duplication formula for the $\Ga$-function,
$$
\Ga(2t)=2^{2t-1}\pi^{-\frac12}\Ga(t)\Ga(t+\tfrac12),
$$
we reduce this to
\begin{gather*}
\frac{2(-1)^{n+m}}\pi\, \sqrt{\dfrac{(2n+a+\frac32)\Ga(n+a+\frac32)\Ga(n+\frac32)}{\Ga(n+a+1)\Ga(n+1)}}\\
\times \sqrt{\dfrac{(2m+a+\frac12)\Ga(m+a+\frac12)\Ga(m+\frac12)}{\Ga(m+a+1)\Ga(m+1)}}.
\end{gather*}
Finally we use the asymptotic formula \cite[\S1.18, (4)]{Er1}
$$
\frac{\Ga(t+\al)}{\Ga(t+\be)}=t^{\al-\be}(1+O(t^{-1})), \qquad t\to+\infty,
$$
where the estimate is uniform with respect to parameters $\al,\be$ provided that they remain bounded. 

This proves \eqref{eq10.F}
\end{proof}

We return to the proof of the theorem. By virtue of Lemma \ref{lemma10.A}, the expression \eqref{eq10.E} takes the form
\begin{equation*}
K^{+,a}(2n+1,2m)=\frac{(-1)^{n+m}}\pi\,\frac{(n+1)\left(1+O\left(\tfrac1{n+1}\right)\right)\left(1+O\left(\tfrac1{m+1}\right)\right)}{(n-m+\frac12)(n+m+a+1)},
\end{equation*}
where the estimates on the right-hand side are uniform with respect to parameter $a$ provided that it is bounded away from $-1$ and $+\infty$ (which we tacitly assume in what follows). 

Further, because
$$
\frac1{n+m+a+1}=\frac{1+O\left(\tfrac1{n+m+1}\right)}{n+m+1},
$$
We can write 
\begin{equation*}
K^{+,a}(2n+1,2m)=\frac{(-1)^{n+m}}\pi\,\frac{(n+1)\left(1+O\left(\tfrac1{n+1}\right)+O\left(\tfrac1{m+1}\right)\right)}{(n-m+\frac12)(n+m+1)}
\end{equation*}
and from this we obtain
$$
|K^{+,a_1}(2n+1,2m)-K^{+,a_2}(2n+1,2m)|=\pi^{-1}\frac{(n+1)\left(O\left(\tfrac1{n+1}\right)+O\left(\tfrac1{m+1}\right)\right)}{(n-m+\frac12)(n+m+1)}
$$
with uniform estimates as before.

We are going to show that the expression on the right-hand side is square summable on $\Z_{\ge0}^2$. It suffices to prove the square summability on $\Z^2_{\ge0}$ separately for 
$$
A(n,m):=\frac1{(n-m+\frac12)(n+m+1)}
$$
and
$$
B(n,m):=\frac{(n+1)}{(m+1)(n-m+\frac12)(n+m+1)}.
$$
For $A(n,m)$ this is evident --- just take  $n-m+\frac12$ and $n+m+1$ as new variables. For $B(n,m)$ we observe that $\dfrac{n+1}{n+m+1}\le1$. This reduces the task to the similar claim about
$$
B'(n,m):=\frac1{(m+1)(n-m+\frac12)},
$$
and this is clear because we can take $m+1$ and $n-m+\frac12$ as new variables. 

We have proved that
$$
\sum_{n,m\in\Z_{\ge0}}(K^{+,a_1}(2n+1,2m)-K^{+,a_2}(2n+1,2m))^2<\infty,
$$
which implies that $K^{+,a_1}-K^{+,a_2}$ is a Hilbert--Schmidt operator. 

It remains to prove that its Hilbert-Schmidt squared norm depends continuously on the parameters $a_1,a_2$. But this follows from the above argument. Indeed, because our bounds are uniform with respect to the parameters, the above sum converges uniformly on $a_1$ and $a_2$ (provided that they are bounded away from  $-1$ and $+\infty$). Finally, we use that evident fact that each summand depends on the parameters continuously. 

This completes the proof of the theorem.
\end{proof}

\bigskip

\textbf{Acknowledgment.} 
I am grateful to Alexander Bufetov for stimulating discussions. I am also grateful to Cesar Cuenca for very helpful comments.

\bigskip


\begin{thebibliography}{AAA}

\bibitem{Araki}
Araki, H.: On quasifree states of CAR and Bogoliubov automorphisms.
Publ. RIMS Kyoto Univ. 6, 385--442 (1970/71)

\bibitem{BKMM-2007}
Baik, J., Kriecherbauer, T.,  McLaughlin, K., Miller, P.: Discrete orthogonal polynomials. Asymptotics and applications.  Annals of Math Studies, 164. Princeton University Press (2007) 

\bibitem{Baker}
Baker, B. M.: Free states of the gauge invariant canonical anticommutation relations. Trans. Amer. Math. Soc. 237, 35--61 (1978) 

\bibitem{BKPV}
Ben Hough, J.,  Krishnapur, M., Peres, Y., Virag, B.: Determinantal processes and independence. Probability Surveys 3, 206--229 (2006) 

\bibitem{Biane}
Biane, P.: Introduction to random walks on noncommutative spaces. In: Quantum Potential Theory, Springer Lecture Notes in Mathematics vol. 1954, 61--116 (2008)

\bibitem{Bor-2011}
Borodin, A.: Determinantal point processes. In: The Oxford Handbook of
Random Matrix Theory. Oxford University Press , New York, pp. 231--249 (2011); arXiv:0911.1153 

\bibitem{BOk-2000}
Borodin, A., Okounkov, A.:  A Fredholm determinant formula for Toeplitz determinants. Integral Equations and Operator Theory 37, 386--396 (2000) 


\bibitem{BOO}
Borodin, A., Okounkov, A., Olshanski, G.:  Asymptotics of Plancherel
measures for symmetric groups.  J. Amer. Math. Soc. 13, 481--515 (2000)

\bibitem{BO-1998}
Borodin, A., Olshanski, G.: Point processes and the infinite symmetric group. Mathematical Research Letters 5,  1--18 (1998)

\bibitem{BO-2000}
Borodin, A., Olshanski, G.:  Distributions on partitions, point processes
and the hypergeometric kernel.  Comm. Math. Phys. 211, 335--358  (2000)


\bibitem{BO-2005}
Borodin, A., Olshanski, G.: Random partitions and the Gamma kernel.
Advances in Math.  194, 141--202  (2005)


\bibitem{BO-2006a}
Borodin, A., Olshanski, G.:  Markov processes on partitions.  Prob.
Theory and Related Fields 135, 84--152 (2006) 

\bibitem{BO-2006b}
Borodin, A., Olshanski, G.:  Meixner polynomials and random partitions. Moscow Math. J.  6, 629--655 (2006)

\bibitem{BO-2007}
Borodin, A., Olshanski, G.:  Asymptotics of Plancherel-type random partitions. J. Algebra 313, 40--60  (2007)

\bibitem{BO-2009}
Borodin, A., Olshanski, G.:  Infinite-dimensional diffusions as limits of random walks on partitions. Probab. Theory Rel. Fields 144, 281--318 (2009).

\bibitem{BO-2017}
Borodin, A., Olshanski, G.: The ASEP and determinantal point processes. Communications in Mathematical Physics 353, 853--903  (2017)

\bibitem{BOS-2006}
Borodin, A., Olshanski, G., Strahov, E.: Giambelli compatible point processes.
Adv. Appl. Math. 37, 209--248  (2006)

\bibitem{BS-2010}
B\"{o}ttcher, A., Spitkovsky, I. M.:  A gentle guide to the basics of two projections theory. 
Linear Algebra Appl. 432, 1412--1459  (2010) 

\bibitem{BR}
O. Bratteli and D. W. Robinson, Operator algebra and quantum statistical mechanics, vol. 2. Second edition, Springer (1997) 

\bibitem{BrownOzawa}
Brown, N. P., Ozawa, N.: $C^*$-algebras and finite-dimensional approximations, Graduate Studies in Mathematics, vol. 88, American Mathematical Society, Providence, RI (2008)

\bibitem{Buf-2015}
Bufetov, A. I.: Action of the group of diffeomorphisms on determinantal measures. Russian Math. Surveys 70, 953--954 (2015) 

\bibitem{Buf-2018}
Bufetov, A. I.:  Quasi-symmetries of determinantal point processes.  Ann. Probab. 46, 956--1003  (2018)

\bibitem{BufOls}
Bufetov, A. I.,  Olshanski, G.: A hierarchy of determinantal point processes with gamma kernels, arXiv:1904.13371 

\bibitem{Dixmier}
Dixmier, J.: Position relative de deux vari\'et\'es et les op\'erateurs lin\'eaires ferm\'ees dans un espace de Hilbert. La Revue Scientifique 86, 387--399 (1948) 

\bibitem{Er1}
Erdelyi, A. (ed.):  Higher transcendental functions. Bateman Manuscript Project, vol. I.  McGraw-Hill,  New York (1953)

\bibitem{Er2}
Erdelyi, A. (ed.):   Higher transcendental functions. Bateman Manuscript Project, vol. II  McGraw-Hill,  New York (1953)

\bibitem{Ghosh}
Ghosh, S.: Determinantal processes and completeness of random exponentials: the critical case. Probability Theory and Related Fields, 163,  643--665  (2015)

\bibitem{GhoshPeres}
Ghosh, S., Peres, Y.:  Rigidity and tolerance in point processes: Gaussiam zeroes and Ginibre eigenvalues. Duke Math. J.  166, 1789--1858  (2017) 

\bibitem{GGK}
Gohberg, I., Goldberg, S., Krupnik, N.:  Traces and determinants of linear operators. Birkh\"{a}user (2000)

\bibitem{GK}
IGohberg, I., Krein, M.:  Introduction to the theory of linear non-self adjoint operators. Transl. AMS (1969) 

\bibitem{Halmos}
Halmos, R. P.: Two subspaces. Trans. Amer. Math.  Soc. 144, 381--389  (1969) 

\bibitem{Kakutani}
Kakutani, S.: On equivalence of infinite product measures. Ann. Math. 2nd Ser. 49, 214--224  (1948) 

\bibitem{KLS-2010}
Koekoek, R., Lesky, P. A., Swarttouw, R. F.:  Hypergeometric orthogonal polynomials and
their q-analogues. Springer (2010)

\bibitem{Konig}
K\"onig, W.:  Orthogonal polynomial ensembles in probability theory. Probability Surveys 2, 385--447 (2005)

\bibitem{Lyons}
Lyons, R.: Determinantal probability measures. Publications Math\'ematiques de l'IH\'ES  98, 167--212  (2003)

\bibitem{Lyt}
Lytvynov, E.: Fermion and boson random point processes as particle distributions
of infinite free fermi and bose gases of finite density, Rev. Math. Phys. 14, 1073--1098  (2002)

\bibitem{LytM}
Lytvynov, E.,  Mei, L.:  On the correlation measure of a family of commuting Hermitian operators with applications to particle densities of the quasifree representations of the CAR and CCR. J. Funct. Anal.  245, 62--88 (2007) 

\bibitem{Mackey}
Mackey, G. W.: Von Neumann and the early days of ergodic theory. In: The legacy of John von Neumann (Proc. Symp. Pure Math. vol. 50).  Amer. Math. Soc.,   pp. 25--38 (1990)

\bibitem{Meyer}
Meyer, P.-A.: Quantum probability for probabilists. Second edition. Lect. Notes Math.  vol. 1538, Springer-Verlag (1995)

\bibitem{Ok-2001a}
Okounkov, A.: Infinite wedge and measures on partitions. Selecta Math. 7, 1--25  (2001)

\bibitem{Ok-2001b}
Okounkov, A.:  SL(2) and z-measures. In: Random matrices and their applications (P. Bleher and A. Its, eds.). MSRI Publications vol. 40, pp. 407--420; arXiv:math/0002135 (2001)

\bibitem{Ols-98}
Olshanski, G.: Point processes and the infinite symmetric group. Part III: Fermion point processes, arXiv:math/9804088.

\bibitem{Ols-2008}
Olshanski, G.: Difference operators and determinantal point processes. Funct. Anal.  Appl. 42, 317--329  (2008)

\bibitem{Ols-2011}
Olshanski, G.:  The quasi-invariance property for the Gamma kernel
determinantal measure. Adv. Math. 226, 2305--2350  (2011)

\bibitem{PS-1970}
Powers, R. T., St{\o}rmer, E.:  Free states of the canonical anticommutation relations. Commun. Math. Phys. 16, 1--33 (1970) 

\bibitem{ST-1}
Shirai, T., Takahashi, Y.:  Random point fields associated with certain Fredholm determinants I: fermion, Poisson and boson point processes. J. Funct. Anal. 205, 414--463  (2003)

\bibitem{ST-2}
Shirai, T., Takahashi, Y.: Random point fields associated with certain Fredholm determinants II: fermion shifts and their ergodic and Gibbs properties. Ann. Prob. 31, 1533-1564  (2003) 

\bibitem{Soshnikov}
Soshnikov, S.:  Determinantal random point fields. Russian Math. Surveys 55, 923--975  (2000)

\bibitem{Stevens}
Stevens, M.: Equivalent symmetric kernels of deterinantal point processes, arXiv:1905.08162.

\bibitem{StrFyo}
Strahov, E., Fyodorov, Y. V.: Universal results for correlations of characteristic polynomials: Riemann-Hilbert
approach. Comm. Math. Phys. 241, 343--382  (2003)

\bibitem{SV-1975}
Str\u{a}til\u{a}, \c{S},  Voiculescu, D.: Representations of AF-algebras and of the group $U(\infty)$. Lect. Notes Math 486. Springer (1975) 

\bibitem{SV-1978}
Str\u{a}til\u{a}, \c{S},  Voiculescu, D.: On a class of KMS states for the unitary group $U(\infty)$. Math. Ann.  235, 87--110  (1978)

\bibitem{Sz}
Szeg\"{o}, G.:  Orthogonal polynomials. American Mathematical Society  (1959)

\bibitem{VK}
Vershik, A. M., Kerov, S. V.: Locally semisimple algebras. Combinatorial theory and the $K_0$-functor. Journal of Soviet Mathematics  38 (2), 1701--1733 (1987) 

\bibitem{Williams}
Williams, D. P. : Crossed products of $C^*$-algebras. Mathematical
Surveys and Monographs, vol. 134, American Mathematical Society, Providence,
RI (2007)

\bibitem{Y}
Yamasaki, Y.: Measures on infinite dimensional spaces. World Scientific (1985) 

\end{thebibliography}
\end{document}